\documentclass[10pt,a4paper]{article}

\usepackage{amsmath, amssymb}
\usepackage{amsthm}
\usepackage[margin=1in]{geometry}
\usepackage{mathtools}
\usepackage[shortlabels]{enumitem}
\usepackage{xcolor}
\usepackage{graphicx}
\usepackage{tikz}
\usetikzlibrary{calc}
\usepackage[hidelinks]{hyperref}
\hypersetup{
  pdftitle={Inverse Robin Spectral Problem for the p-Laplace Operator},
  pdfauthor={Farid Bozorgnia and Olimjon Eshkobilov}
}

\newtheorem{theorem}{Theorem}[section]
\newtheorem{lemma}[theorem]{Lemma}
\newtheorem{proposition}[theorem]{Proposition}
\newtheorem{corollary}[theorem]{Corollary}
\theoremstyle{definition}
\newtheorem{definition}[theorem]{Definition}
\newtheorem{remark}[theorem]{Remark}
\newtheorem{assumption}{Assumption}

\title{Inverse Robin Spectral Problem for the $p$-Laplace Operator}
\author{Farid Bozorgnia, Olimjon Eshkobilov\\
\small Department of Mathematics, New Uzbekistan University\\
\small Email: f.bozorgnia@newuu.uz, \quad Email: o.eshqobilov@newuu.uz}
\date{}

\begin{document}
\maketitle

\begin{abstract}
We study an inverse Robin spectral problem for the $p$-Laplacian with mixed
Dirichlet--Robin boundary conditions. The unknown coefficient lies on an
inaccessible Robin portion, and the data are the principal eigenvalue and
nonlinear conormal flux on an accessible Dirichlet patch. For a coating of
thickness $\varepsilon\rho$ and conductivity $\varepsilon^{p-1}$, we prove
convergence of the principal eigenvalue and strong $W^{1,p}$ convergence of
the eigenfunctions to an effective Robin problem with coefficient 
$h=\rho^{-(p-1)}$. The Rayleigh characterization gives a two-sided comparison
inequality, ordered uniqueness and weighted $L^1$ stability, and Fr\'echet
differentiability of the eigenvalue map for $1<p<\infty$. For nonordered
coefficients and $p\ge2$, the eigenvalue and accessible flux determine the
coefficient under the stated critical-set connectivity and local
boundary-regularity hypotheses. Differentiability of the principal
eigenfunction branch is proved for $p=2$ and imposed as a structural hypothesis
for $p>2$; together with a quantitative linearized estimate, it yields
conditional local stability. On finite-dimensional classes, branch
differentiability and linearized injectivity yield local Lipschitz stability
(with branch differentiability automatic when $p=2$), while monotone
one-parameter families are identifiable from the principal eigenvalue.
\end{abstract}



\section{Introduction}\label{sec:introduction}

Inverse spectral and boundary-value problems seek to recover unknown
properties of a physical system from indirect observations such as
eigenvalues, boundary traces, or flux measurements. Their mathematical
foundations go back to classical questions in spectral geometry and inverse
Sturm--Liouville theory, including the works of Kac~\cite{kac1966} and
Gel'fand--Levitan~\cite{gelfand1951}. Related inverse boundary problems arise
in non-destructive testing~\cite{ikehata2000,bacchelli2006}, electrical
impedance tomography~\cite{nachman1996,astala2006}, and corrosion
detection~\cite{alessandrini2003}. For techniques based on singular solutions
and boundary measurements in inverse conductivity problems we refer
to~\cite{alessandrini1990singular}.

Robin boundary conditions are particularly relevant when the boundary of a
body interacts with an exterior medium or is separated from it by a thin
layer. Effective boundary conditions of this type have a long history in the 
linear setting. For a layer of thickness of order $\varepsilon$ and conductivity
$\sigma_\varepsilon$, Friedman~\cite{Fr} showed that the critical scaling
$\sigma_\varepsilon/\varepsilon\to\alpha\in(0,\infty)$ leads, in the limit,
to a Robin boundary condition for the associated two-phase principal
eigenvalue problem, whereas
$\sigma_\varepsilon/\varepsilon\to0$ produces perfect insulation.
Variational formulations of the same limit were developed by Brezis,
Caffarelli and Friedman~\cite{BrezisEtAl80} and by Acerbi and
Buttazzo~\cite{AcerbiButtazzo86}, and the resulting surface energy is the
starting point of the optimal insulation problems of Bucur, Buttazzo and
Nitsch~\cite{BucurEtAl17} and, for a layer of variable thickness, of
Acampora, Cristoforoni, Nitsch and Trombetti~\cite{AcamporaEtAl24}. A
refinement of Friedman's asymptotics that tracks the geometry of the interface
is due to Yachimura~\cite{Yachimura19}. For the $p$-Laplacian, a
$\Gamma$-limit of the insulation energy for a layer with a $p$-dependent
thickness scaling has recently been obtained by Della Pietra and
Oliva~\cite{DellaPietraOliva26}. The thin-layer result proved in
Section~\ref{sec:coating} below is of a different nature: it is a spectral
statement for a layer of variable thickness; it produces strong
$W^{1,p}(\Omega)$ convergence of the eigenfunctions, and it identifies the
effective coefficient explicitly.

The purpose of the present paper is to investigate the corresponding mechanism
and the associated inverse problem for the nonlinear $p$-Laplace operator
\[
-\Delta_p u:=-\operatorname{div}\bigl(|\nabla u|^{p-2}\nabla u\bigr),
\qquad 1<p<\infty .
\]
The $p$-Laplacian is a basic model for nonlinear diffusion and power-law
constitutive relations and arises, for example, in models of non-Newtonian
fluids and nonlinear transport~\cite{bird1987}. For $p=2$ it reduces to the
classical Laplacian, whereas for $p\neq2$ it is genuinely nonlinear and may be
degenerate or singular at critical points of the solution.

The inverse theory for the nonlinear problem is substantially more delicate
than for $p=2$. The superposition principle is unavailable, there is no linear
orthogonal spectral decomposition, and linearization around a $p$-Laplace
eigenfunction produces coefficients depending on the gradient, which for $p>2$
lose uniform ellipticity at critical points of the eigenfunction. These
features make classical unique-continuation and stability arguments
considerably more restrictive than in the linear setting, where inverse Robin
problems have been extensively studied; see for
example~\cite{sy2020,Is,Chaabane04} and the references therein.

\subsection{Problem setting}\label{sec:problem-setting}

Unless explicitly stated otherwise, $\Omega\subset\mathbb R^n$, $n\ge2$, is a
bounded connected domain with $C^2$ boundary, and $\Gamma_D$, $\gamma$ are
nonempty, disjoint, relatively open subsets of $\partial\Omega$. We write
\begin{equation}\label{eq:interface-zeta}
\zeta:=\overline{\Gamma_D}\cap\overline{\gamma}
\end{equation}
for the interface between the two boundary portions, and we impose the
standing geometric hypothesis
\begin{equation}\label{eq:geometric-hypothesis}
\partial\Omega=\Gamma_D\cup\gamma\cup\zeta,
\qquad
\mathcal H^{n-1}(\zeta)=0 .
\end{equation}
The exception is Section~\ref{sec:coating}, which deliberately treats the
fully coated geometry $\gamma=\partial\Omega$, $\Gamma_D=\emptyset$; the
results quoted there are stated in a form that covers this case, see
Remark~\ref{rem:pure-robin} and Lemma~\ref{lem:regularity}.

We refer to $\Gamma_D$ as the \emph{accessible} boundary portion and to
$\gamma$ as the \emph{inaccessible} portion carrying the unknown Robin
coefficient. Set
\begin{equation}\label{eq:V-space}
V:=W^{1,p}_{\Gamma_D}(\Omega)
:=\bigl\{v\in W^{1,p}(\Omega):\ Tv=0\ \text{on }\Gamma_D\bigr\},
\end{equation}
where $T$ denotes the Sobolev trace operator. For a nonnegative Robin
coefficient $h$ on $\gamma$ we consider the principal eigenvalue problem
\begin{equation}\label{eq:p-laplace-robin}
\begin{cases}
-\Delta_p u=\lambda|u|^{p-2}u & \text{in }\Omega,\\[2mm]
u=0 & \text{on }\Gamma_D,\\[2mm]
|\nabla u|^{p-2}\nabla u\cdot\nu+h|u|^{p-2}u=0 & \text{on }\gamma,
\end{cases}
\end{equation}
where $\nu$ is the outward unit normal. We restrict attention to the principal
eigenpair and normalize
\begin{equation}\label{eq:normalization-intro}
\int_\Omega|u|^p\,dx=1,
\qquad
u>0\ \text{ in }\Omega .
\end{equation}
The principal eigenvalue admits a variational characterization
\begin{equation}\label{eq:rayleigh-introduction}
\lambda_1(h)
=\inf_{\substack{v\in V\\ v\neq0}}
\frac{\displaystyle\int_\Omega|\nabla v|^p\,dx+\int_\gamma h|v|^p\,d\sigma}
{\displaystyle\int_\Omega|v|^p\,dx},
\end{equation}
whose elementary properties are collected in
Section~\ref{sec:principal-eigenvalue}.

The inverse data are $\lambda_1(h)$ together with the nonlinear conormal flux
\begin{equation}\label{eq:flux-introduction}
q_h:=|\nabla u(h)|^{p-2}\nabla u(h)\cdot\nu .
\end{equation}
In the natural Sobolev formulation, $q_h$ is the conormal trace of the vector
field $|\nabla u|^{p-2}\nabla u$, made precise in Section~\ref{sec:stability};
on compact subportions of $\Gamma_D$ the boundary regularity of
Lemma~\ref{lem:regularity} gives it the pointwise meaning written above. Since
$u(h)=0$ on $\Gamma_D$ by construction, it is the flux and not the Dirichlet
trace of the eigenfunction that carries the nontrivial boundary information.
The inverse Robin spectral problem considered here is therefore the following.

\medskip
\noindent\emph{To what extent does the pair consisting of $\lambda_1(h)$ and
the flux $q_h$ on the accessible boundary determine the Robin coefficient $h$
on $\gamma$, and how stable is this determination with respect to
perturbations of the measured data?}
\medskip

\subsection{Main results and scope}\label{sec:intro-main-results}

\paragraph{Thin-coating limit.}
Let the exterior coating have local thickness $\varepsilon\rho(\xi)$ with
$\rho>0$ and conductivity $\sigma_\varepsilon=\varepsilon^{p-1}$, a scaling
that balances the thickness of the layer against the $(p-1)$-homogeneity of
the nonlinear normal flux and is of critical order in the sense of
Remark~\ref{rem:criticality}. If $\Lambda_1(\varepsilon)$ and
$\Phi_\varepsilon$ denote the principal eigenvalue and the normalized positive
eigenfunction of the two-phase problem, then
\begin{equation}\label{eq:intro-thin-limit}
\Lambda_1(\varepsilon)\longrightarrow\mu_1,
\qquad
\Phi_\varepsilon|_\Omega\longrightarrow u_1
\quad\text{strongly in }W^{1,p}(\Omega),
\end{equation}
where $(\mu_1,u_1)$ is the normalized positive principal eigenpair of the
effective Robin problem on $\Omega$ with coefficient
\begin{equation}\label{eq:intro-effective-h}
h(\xi)=\rho(\xi)^{-(p-1)} .
\end{equation}

\paragraph{Uniqueness for ordered coefficients.}
Since \eqref{eq:rayleigh-introduction} exhibits $\lambda_1$ as an infimum of
functionals affine in $h$, one obtains, for all $p\in(1,\infty)$ and all
$n\ge2$, the two-sided comparison
\begin{equation}\label{eq:intro-sandwich}
\int_\gamma(h_2-h_1)\,u_2^{\,p}\,d\sigma
\;\le\;
\lambda_1(h_2)-\lambda_1(h_1)
\;\le\;
\int_\gamma(h_2-h_1)\,u_1^{\,p}\,d\sigma .
\end{equation}
Three consequences follow, none requiring a nondegeneracy or invertibility
hypothesis: ordered coefficients with equal principal eigenvalues coincide;
in the ordered regime $h_2-h_1$ obeys a Lipschitz estimate in a weighted
$L^1(\gamma)$ norm; and the derivative
\eqref{eq:intro-eigenvalue-derivative} below exists. In the thin-layer model
$h_1\le h_2$ means $\rho_1\ge\rho_2$, so the limiting principal eigenvalue
distinguishes any two pointwise ordered thickness profiles.

\paragraph{Conditional uniqueness in the general case.}
Without an ordering assumption we use the flux data as well and take $p\ge2$.
With $F(z)=|z|^{p-2}z$ and $(\lambda_1(h_j),u_j)$ the normalized principal
eigenpairs, the difference of the two eigenfunctions solves a linear equation
with the path coefficient
\begin{equation}\label{eq:intro-path-matrix}
\overline A(x)
=\int_0^1 DF\bigl((1-t)\nabla u_2(x)+t\nabla u_1(x)\bigr)\,dt .
\end{equation}
Under a connectedness hypothesis on the critical sets
$\Sigma_j:=\{\nabla u_j=0\}$, made precise in Section~\ref{sec:uniqueness},
equality of the principal eigenvalues together with equality of the fluxes on
the accessible portion implies $h_1=h_2$ on $\gamma$. The hypothesis is
dimension dependent: in the plane measurable uniformly elliptic coefficients
suffice for propagation, so only $\Sigma_1\cap\Sigma_2$ must be avoided,
whereas for $n\ge3$ the Lipschitz regularity we can verify is available only
away from $\Sigma_1\cup\Sigma_2$.

\paragraph{Differentiability and conditional stability.}
By Proposition~\ref{prop:lambda-remainder}, the eigenvalue map is Fr\'echet
differentiable at any reference coefficient $h_0$, with no
linearized-invertibility hypothesis, and
\begin{equation}\label{eq:intro-eigenvalue-derivative}
\lambda_1'(h_0)[\vartheta]
=\int_\gamma \vartheta\,u_0^{\,p}\,d\sigma .
\end{equation}
The eigenvalue datum thus contributes a single scalar functional at the
linearized level, and the remaining information on $h$ must come from the
boundary flux. For the eigenfunction branch, differentiability is proved when
$p=2$; for $p>2$ the natural linearization is degenerate and the standard
implicit-function framework is unavailable, see
Remark~\ref{rem:no-isomorphism}, so we retain it as a structural hypothesis.
Together with a quantitative stability estimate for the linearized inverse
map, it yields in Section~\ref{sec:stability} the local conditional estimate
\begin{equation}\label{eq:intro-stability-form}
\|h-h_0\|_{L^2(\gamma)}
\le C M^{1-\alpha}
\bigl(\delta+\widetilde\omega(t)t\bigr)^\alpha,
\qquad
t=\|h-h_0\|_{C^1(\overline\gamma)}\le M,
\end{equation}
with $\alpha\in(0,1)$, $\delta$ the discrepancy between the nonlinear
measurements, and $\widetilde\omega(s)\to0$ as $s\to0$. In the full
infinite-dimensional coefficient class the estimate therefore remains
conditional rather than an unconditional data-only H\"older estimate. On
finite-dimensional classes the remainder can be absorbed, and for $p=2$, where
both the branch differentiability and the injectivity of the linearized
measurement map are proved, this gives unconditional local Lipschitz
stability.

The paper is organized as follows. Section~\ref{sec:preliminaries} collects the
variational, regularity and unique-continuation background, together with a
two-sided bound for the averaged Jacobian of $F$ along a segment and a
regularity lemma valid where the gradient does not vanish.
Section~\ref{sec:coating} studies the nonlinear thin-coating model.
Section~\ref{Inv} treats the inverse problem, from boundary non-degeneracy and
the unconditional consequences of the variational characterization through the
conditional uniqueness and stability theorems, and concludes in
Section~\ref{sec:verification} with the status of the structural hypotheses,
the spherical-shell obstruction to critical-set connectivity, and what is
verified unconditionally when $p=2$, including on a non-radial ellipsoid.

\section{Background and Preliminaries}\label{sec:preliminaries}

Throughout this section $1<p<\infty$ unless a more restrictive range is
stated. Recall the nonlinear vector field
\[
F:\mathbb R^n\to\mathbb R^n,
\qquad
F(z):=|z|^{p-2}z .
\]
The letter $F$ is reserved for this field throughout
the paper.  
Reserving $\xi$ for boundary points, we use $z,\eta,a,b$ for vectors of
$\mathbb R^n$.   For $z\neq0$ and $\eta\in\mathbb R^n$, we have  
\begin{equation}\label{eq:DF-matrix}
DF(z)\eta=|z|^{p-2}\eta+(p-2)|z|^{p-4}(z\cdot\eta)z,
\qquad
DF(z)=|z|^{p-2}I+(p-2)|z|^{p-4}z\,z^{\top} .
\end{equation}
The matrix $DF(z)$ is symmetric with eigenvalues $|z|^{p-2}$ of multiplicity
$n-1$ and $(p-1)|z|^{p-2}$ in the direction of $z$, whence
\begin{equation}\label{eq:DF-two-sided}
\min\{1,p-1\}|z|^{p-2}|\eta|^2
\le DF(z)\eta\cdot\eta
\le\max\{1,p-1\}|z|^{p-2}|\eta|^2 .
\end{equation}
For $p\ge2$ the field $F$ is continuously differentiable on all of
$\mathbb R^n$, with $DF(0)=I$ when $p=2$ and $DF(0)=0$ when $p>2$; for
$1<p<2$ it is not differentiable at the origin. Similarly
\[
G(s):=|s|^{p-2}s,
\]
 satisfies $G'(s)=(p-1)|s|^{p-2}$ for $s\neq0$, and $G$ is
$C^1$ on $\mathbb R$ when $p\ge2$. By \eqref{eq:DF-two-sided} the
linearization of the $p$-Laplacian degenerates for $p>2$ where $\nabla u=0$
and becomes singular for $1<p<2$; since the inverse analysis of
Section~\ref{Inv} uses classical linearization, we restrict there to $p\ge2$
from Section~\ref{sec:linearized} onward.

The map $F$ is strictly monotone: for all $a,b\in\mathbb R^{n}$,
$\bigl(F(b)-F(a)\bigr)\cdot(b-a)\ge 0$, with equality if and only if $a=b$.
Quantitatively,
\begin{equation}\label{eq:monotonicity}
\bigl(F(b)-F(a)\bigr)\cdot(b-a)
\ge
\begin{cases}
2^{\,2-p}|b-a|^p, & p\ge2,\\[2mm]
(p-1)(|a|+|b|)^{p-2}|b-a|^2,
& 1<p<2,\quad (a,b)\neq(0,0);
\end{cases}
\end{equation}
see \cite[Ch.~12]{lindqvist2017}. 

Two further properties of $DF$ will be used. The first is a quantitative
increment estimate, needed for the regularity of the path coefficient in
Section~\ref{sec:linearized}.

\begin{lemma}\label{lem:DF-increments}
Let $p>2$. There exists $C_p<\infty$ such that, for all
$a,b\in\mathbb R^n$ with $(a,b)\neq(0,0)$,
\begin{equation}\label{eq:DF-increment-max}
\|DF(a)-DF(b)\|
\le
C_p\max\{|a|,|b|\}^{p-3}|a-b|.
\end{equation}
Consequently,
\begin{equation}\label{eq:DF-increment-cases}
\|DF(a)-DF(b)\|
\le
\begin{cases}
C_p|a-b|^{p-2}, & 2<p\le3,\\[1mm]
C_p(|a|+|b|)^{p-3}|a-b|, & p\ge3.
\end{cases}
\end{equation}
For $p=2$, $DF\equiv I$.
\end{lemma}

\begin{proof}
Let $p>2$ and set
\[
M:=\max\{|a|,|b|\}>0,
\qquad
a':=\frac aM,\qquad b':=\frac bM .
\]
By the homogeneity $DF(tz)=t^{p-2}DF(z)$, $t>0$,
\begin{equation}\label{eq:DF-scaling}
\|DF(a)-DF(b)\|
=
M^{p-2}\|DF(a')-DF(b')\| ,
\end{equation}
and we may assume $|a'|=1$, $|b'|\le1$. We claim that
\begin{equation}\label{eq:DF-normalized}
\|DF(a')-DF(b')\|\le C_p|a'-b'| .
\end{equation}
Three cases are distinguished. If the segment $[a',b']$ is contained in
$\{|z|\ge\frac14\}$, then, since also $|z|\le1$ on the segment, the explicit
formula \eqref{eq:DF-matrix} gives $\|D^2F(z)\|\le C_p|z|^{p-3}\le C_p$
there, and \eqref{eq:DF-normalized} follows from the mean value theorem. If
the segment meets a point $v_0$ with $|v_0|<\frac14$, then
$|a'-b'|\ge|a'-v_0|\ge\frac34$, while
$\|DF(a')\|+\|DF(b')\|\le2(p-1)$ because $|a'|,|b'|\le1$ and $p>2$; hence
\eqref{eq:DF-normalized} holds after enlarging $C_p$. The same trivial bound
covers the remaining case $b'=0$.

Substituting \eqref{eq:DF-normalized} into \eqref{eq:DF-scaling} and using
$|a'-b'|=|a-b|/M$ proves \eqref{eq:DF-increment-max}. If $p\ge3$, then
$M^{p-3}\le(|a|+|b|)^{p-3}$. If $2<p\le3$, then $|a-b|\le2M$ and therefore
$M^{p-3}|a-b|\le2^{3-p}|a-b|^{p-2}$. These give
\eqref{eq:DF-increment-cases}.
\end{proof}

The second compares the average of $DF$ along a segment with the larger of the
two endpoints, and shows in particular that this average is uniformly
elliptic exactly when the larger endpoint is bounded above and away from zero.

\begin{lemma}\label{lem:path-matrix}
Let $p\ge2$, let $a,b\in\mathbb R^n$, and set
\[
M:=\max\{|a|,|b|\},
\qquad
\overline A
:=
\int_0^1
DF\bigl((1-t)a+tb\bigr)\,dt .
\]
If $p=2$, then $\overline A=I$. If $p>2$, then
\begin{equation}\label{eq:segment-integral}
2^{-p}M^{p-2}
\le
\int_0^1
\bigl|(1-t)a+tb\bigr|^{p-2}\,dt
\le
M^{p-2},
\end{equation}
and consequently, for every $\eta\in\mathbb R^n$,
\begin{equation}\label{eq:path-two-sided}
2^{-p}M^{p-2}|\eta|^2
\le
\overline A\eta\cdot\eta
\le
(p-1)M^{p-2}|\eta|^2 .
\end{equation}
\end{lemma}

\begin{proof}
The case $p=2$ is immediate. Let $p>2$. The upper bound in
\eqref{eq:segment-integral} follows from
$|(1-t)a+tb|\le(1-t)|a|+t|b|\le M$.
For the lower bound we may assume, by the symmetry $t\mapsto1-t$, that
$|a|=M$. For $0\le t\le\frac14$,
\[
|(1-t)a+tb|
\ge
|a|-t|b-a|
\ge
M-2tM
\ge
\frac M2 ,
\]
whence
\[
\int_0^1|(1-t)a+tb|^{p-2}\,dt
\ge
\frac14\Bigl(\frac M2\Bigr)^{p-2}
=
2^{-p}M^{p-2}.
\]
Finally, \eqref{eq:path-two-sided} follows by integrating
\eqref{eq:DF-two-sided} along the segment, using $\min\{1,p-1\}=1$ and
$\max\{1,p-1\}=p-1$ for $p\ge2$.
\end{proof}

Applied pointwise to vector fields $a=a(x)$ and $b=b(x)$ on a set $E$,
Lemma~\ref{lem:path-matrix} shows that, for $p>2$, $\overline A(x)$ is
uniformly elliptic on $E$ precisely when
\[
0<m_0
\le
\max\{|a(x)|,|b(x)|\}
\le
M_0<\infty
\qquad (x\in E).
\]
For $p=2$, uniform ellipticity is automatic since $\overline A\equiv I$.

\subsection{Principal eigenvalue with mixed boundary conditions}
\label{sec:principal-eigenvalue}

We first specify the weak formulation of \eqref{eq:p-laplace-robin}.

\begin{definition}\label{def:weak_formulation}
A pair $(\lambda,u)\in\mathbb R\times V$ with $u\not\equiv0$ is a \emph{weak
eigenpair} of \eqref{eq:p-laplace-robin} if
\begin{equation}\label{eq:weak_formulation}
\int_\Omega|\nabla u|^{p-2}\nabla u\cdot\nabla\varphi\,dx
+\int_\gamma h|u|^{p-2}u\,\varphi\,d\sigma
=\lambda\int_\Omega|u|^{p-2}u\,\varphi\,dx
\qquad\text{for all }\varphi\in V .
\end{equation}
\end{definition}

The next proposition states the existence, positivity and simplicity of the
principal eigenpair, and shows that the positivity and normalization
conditions single it out uniquely.

\begin{proposition}\label{prop:principal_eigenvalue}
Let $\Omega$, $\Gamma_D$, and $\gamma$ satisfy the standing assumptions of
Section~\ref{sec:problem-setting}, and let $h\in L^\infty(\gamma)$ with
$h\ge0$. Then the infimum in \eqref{eq:rayleigh-introduction} is attained,
$\lambda_1(h)>0$, and there is a normalized eigenfunction $u_1\in V$ with
$\|u_1\|_{L^p(\Omega)}=1$ and $u_1>0$ in $\Omega$. Moreover the principal
eigenvalue is simple in the strong sense that
\begin{equation}\label{eq:full-simplicity}
\text{every weak eigenfunction associated with }\lambda_1(h)
\text{ is a scalar multiple of }u_1 .
\end{equation}
Consequently $u_1>0$ and $\|u_1\|_{L^p(\Omega)}=1$ select a unique principal
eigenfunction, denoted $u(h)$.
\end{proposition}

\begin{proof}
Since $h\ge0$, the Rayleigh quotient in \eqref{eq:rayleigh-introduction} is
bounded below by the optimal constant in the Poincar\'e inequality on $V$,
which is positive because $\mathcal H^{n-1}(\Gamma_D)>0$; hence
$\lambda_1(h)>0$. Existence of a minimizer follows from the direct method,
using that inequality, the compactness of
$W^{1,p}(\Omega)\hookrightarrow L^p(\Omega)$ and of the trace
$W^{1,p}(\Omega)\to L^p(\partial\Omega)$; the boundary term is weakly
continuous by the latter. Replacing a minimizer by its absolute value and
applying the strong maximum principle \cite{Vazquez84} yields a normalized
eigenfunction $u_1>0$ in $\Omega$.

Let now $v$ be any weak eigenfunction associated with $\lambda_1(h)$. Testing
\eqref{eq:weak_formulation} with $\varphi=v$ shows that $v$ realizes the
infimum in \eqref{eq:rayleigh-introduction}, hence so does $|v|$, which is
therefore itself an eigenfunction and is positive in $\Omega$ by the strong
maximum principle. Applying Picone's identity for the $p$-Laplacian
\cite{AllegrettoHuang1998} to the pair $u_1,|v|$, through the admissible test
functions $|v|^{p}(u_1+\epsilon)^{1-p}$ and $u_1^{p}(|v|+\epsilon)^{1-p}$ and
letting $\epsilon\downarrow0$, gives $|v|=c\,u_1$ for some $c>0$; the Robin
contributions cancel because $u_1$ and $|v|$ satisfy the same boundary
condition with coefficient $h$. Finally, $v$ is continuous in $\Omega$ by
Lemma~\ref{lem:Linfty} and Lemma~\ref{lem:regularity}(i), whose proofs use
neither simplicity nor the present proposition. Since $|v|=c\,u_1>0$ on the
connected open set $\Omega$, the function $v$ has a constant sign there, and
therefore $v=\pm c\,u_1$. This proves \eqref{eq:full-simplicity}, and
positivity together with $\|u_1\|_{L^p(\Omega)}=1$ selects a unique principal
eigenfunction, denoted $u(h)$.
\end{proof}

The following remark records the version of that proposition needed for the
fully coated geometry of Section~\ref{sec:coating}, in which the Dirichlet
portion is empty.

\begin{remark}\label{rem:pure-robin}
Proposition~\ref{prop:principal_eigenvalue} remains valid when
$\Gamma_D=\emptyset$, that is for the purely Robin problem on
$\partial\Omega=\gamma$, provided $\operatorname*{ess\,inf}_\gamma h>0$. Then
$V=W^{1,p}(\Omega)$ and coercivity follows from the equivalence of norms
\begin{equation}\label{eq:equivalent-norm}
\|v\|_{W^{1,p}(\Omega)}^p
\le C(\Omega,p)\Bigl(\int_\Omega|\nabla v|^p\,dx
+\int_{\partial\Omega}|v|^p\,d\sigma\Bigr),
\qquad v\in W^{1,p}(\Omega),
\end{equation}
valid on bounded Lipschitz domains~\cite{Necas67}, which is the case used in
Section~\ref{sec:coating}; there $h=\rho^{-(p-1)}$ with $\rho\in C^2(\Gamma)$
positive on the compact set $\Gamma$, so that $\inf_\gamma h>0$.  The same conclusions hold for the two-phase problem \eqref{eq:two-phase}. For
each fixed $\varepsilon>0$ the coefficient $\sigma_\varepsilon$ in
\eqref{eq:sigma-eps} satisfies $\varepsilon^{p-1}\le\sigma_\varepsilon\le1$,
so the energy $\Phi\mapsto\int_{\Omega_\varepsilon}\sigma_\varepsilon
|\nabla\Phi|^p\,dx$ is coercive on $W^{1,p}_0(\Omega_\varepsilon)$ by the
Poincar\'e inequality, which requires no regularity of
$\partial\Omega_\varepsilon$, and the direct method applies. The operator
$\operatorname{div}(\sigma_\varepsilon|\nabla\Phi|^{p-2}\nabla\Phi)$ satisfies
the standard structure conditions with ellipticity constants
$\varepsilon^{p-1}$ and $1$, so the Harnack inequality and the strong maximum
principle for quasilinear equations with measurable coefficients of $p$-growth
give positivity of the first eigenfunction in the connected set
$\Omega_\varepsilon$. Simplicity follows from Picone's identity
\cite{AllegrettoHuang1998}, which holds verbatim with the weight
$\sigma_\varepsilon$ in the leading term, the boundary terms now vanishing
because both functions have zero trace on $\partial\Omega_\varepsilon$.
\end{remark}

The next lemma establishes boundedness of the principal eigenfunction by a
standard Moser iteration. It is stated for both geometries, since it is used
in Section~\ref{sec:coating} in the fully coated case.
\begin{lemma}\label{lem:Linfty}
Assume either the mixed-boundary geometry of
Section~\ref{sec:problem-setting}, or
\[
\Gamma_D=\emptyset,\qquad\gamma=\partial\Omega,\qquad V=W^{1,p}(\Omega).
\]
Let $h\in L^\infty(\gamma)$ with $h\ge0$, and let $u\in V$ be a weak
eigenfunction of \eqref{eq:p-laplace-robin} associated with $\lambda_1(h)$,
normalized by $\|u\|_{L^p(\Omega)}=1$; in particular, $u$ may be a normalized
principal eigenfunction. Then $u\in L^\infty(\Omega)$, with a bound depending
only on $n$, $p$, $\Omega$ and $\lambda_1(h)$.
\end{lemma}

\begin{proof}
We may assume $u\ge0$. For $k>0$ and $\kappa\ge0$ put $u_k:=\min\{u,k\}$ and
test \eqref{eq:weak_formulation} with $\varphi:=u\,u_k^{\,\kappa}\in V$, which
is admissible because $s\mapsto s\min\{s,k\}^{\kappa}$ is Lipschitz on
$[0,\infty)$ and vanishes at $0$; for $\kappa=0$ this is $\varphi=u$.
Since $h\ge0$ and $u\ge0$, the boundary term is nonnegative and may be
discarded, which gives
\[
\int_\Omega u_k^{\,\kappa}|\nabla u|^p\,dx
+\kappa\int_{\{u<k\}}u\,u_k^{\,\kappa-1}|\nabla u|^p\,dx
\ \le\ \lambda_1(h)\int_\Omega u^{p}\,u_k^{\,\kappa}\,dx .
\]
Writing $w_k:=u\,u_k^{\,\kappa/p}$ and using
\[
|\nabla w_k|^p\le C(p)(1+\kappa)^{p-1}
\bigl(u_k^{\,\kappa}|\nabla u|^p
+\kappa\,u\,u_k^{\,\kappa-1}\mathbf 1_{\{u<k\}}|\nabla u|^p\bigr),
\]
the left-hand side above controls $\|\nabla w_k\|_{L^p(\Omega)}^p$ up to the
factor $C(p)(1+\kappa)^{p-1}$. Since moreover
$\|w_k\|_{L^p(\Omega)}^p=\int_\Omega u^{p}u_k^{\,\kappa}\,dx$, we obtain
\[
\|w_k\|_{W^{1,p}(\Omega)}^p
\le C(1+\kappa)^{p-1}\bigl(1+\lambda_1(h)\bigr)
\int_\Omega u^{p}u_k^{\,\kappa}\,dx ;
\]
the $L^p$ term is needed in the fully coated case $\Gamma_D=\emptyset$, where
no Poincar\'e inequality is available. For $p<n$ the Sobolev embedding
$W^{1,p}(\Omega)\hookrightarrow L^{p^*}(\Omega)$, $p^*=np/(n-p)$, then yields
\[
\|w_k\|_{L^{p^*}(\Omega)}^p
\le C(1+\kappa)^{p-1}\bigl(1+\lambda_1(h)\bigr)
\int_\Omega u^{p}u_k^{\,\kappa}\,dx .
\]
Letting $k\to\infty$, the right-hand side being finite at each stage by the
preceding step, and iterating along the exponents defined by
$p+\kappa_0=p$ and $p+\kappa_{j+1}=(p+\kappa_j)\,p^*/p$, produces bounds in
successively higher $L^q$ with constants whose product converges, and hence
$\|u\|_{L^\infty(\Omega)}\le C$. When $p=n$ the same iteration is run with a
fixed $q>n$ in place of $p^*$, using $W^{1,n}(\Omega)\hookrightarrow
L^q(\Omega)$. When $p>n$ the conclusion follows directly from
$W^{1,p}(\Omega)\hookrightarrow L^\infty(\Omega)$ and the bound
$\|u\|_{W^{1,p}(\Omega)}^p\le 1+\lambda_1(h)$ furnished by the energy
identity and the normalization. This is the standard Moser scheme for
quasilinear equations with nonlinear boundary conditions; see
\cite{Lieberman88,Winkert10}.
\end{proof}

Since the Rayleigh quotient gives
$\lambda_1(h)\le C(\Omega,p)\bigl(1+\|h\|_{L^\infty(\gamma)}\bigr)$, the bound
of Lemma~\ref{lem:Linfty} depends on $h$ only through
$\|h\|_{L^\infty(\gamma)}$.

\subsection{Boundary behavior and regularity}
\label{sec:regularity}

We first record the boundary point principle used later to obtain
nondegeneracy of the gradient on the Dirichlet measurement boundary.

\begin{lemma}\label{lem:hopf}
Let $1<p<\infty$, let $\Omega$ be a bounded $C^2$ domain, and let
$x_0\in\partial\Omega$. Suppose that $u\in C^1(\overline\Omega\cap U)$ in a
neighborhood $U$ of $x_0$,
\[
u>0\quad\text{in }\Omega\cap U,
\qquad
u(x_0)=0,
\]
and $-\Delta_pu=\lambda u^{p-1}$ in $\Omega\cap U$ for some $\lambda\ge0$.
Then $\partial_\nu u(x_0)<0$.
\end{lemma}

\begin{proof}
Since $-\Delta_pu=\lambda u^{p-1}\ge0$ in $\Omega\cap U$, the boundary point
principle for the $p$-Laplacian applies; see
\cite{Vazquez84,Pucci-Serrin07}. The $C^2$ regularity of $\partial\Omega$
provides the interior sphere condition at $x_0$, and therefore
$\partial_\nu u(x_0)<0$.
\end{proof}

We next collect the interior and boundary regularity needed below. The
regularity is local away from the interface
$\zeta=\overline{\Gamma_D}\cap\overline\gamma$, where the boundary condition
changes type.

\begin{lemma}\label{lem:regularity}
Let $\Omega$ have $C^2$ boundary and assume either the mixed-boundary
geometry of Section~\ref{sec:problem-setting}, or
\[
\Gamma_D=\emptyset,\qquad
\gamma=\partial\Omega,\qquad
\zeta=\emptyset,\qquad
V=W^{1,p}(\Omega).
\]
Let $h\in C^1(\overline\gamma)$ with $h\ge0$, and let
$u\in V\cap L^\infty(\Omega)$ be a weak solution of
\eqref{eq:p-laplace-robin} for some $\lambda\in\mathbb R$. Then:
\begin{enumerate}[(i)]
\item there exists $\beta\in(0,1)$ such that
$u\in C^{1,\beta}_{\mathrm{loc}}(\Omega)$;

\item for every compact set $K\subset\overline\Omega\setminus\zeta$ there
exist $\beta_K\in(0,1)$ and $C_K<\infty$ such that
$\|u\|_{C^{1,\beta_K}(K)}\le C_K$;

\item in particular, $\nabla u$ extends continuously to every compact
subportion $\Gamma_D'\Subset\Gamma_D$ and $\gamma'\Subset\gamma$;

\item if $\zeta=\emptyset$, then $u\in C^{1,\beta}(\overline\Omega)$ for some
$\beta\in(0,1)$.
\end{enumerate}
The constants depend only on $n$, $p$, the $C^2$ character of
$\partial\Omega$, $\lambda$, $\|u\|_{L^\infty(\Omega)}$,
$\|h\|_{C^1(\overline\gamma)}$ and $\operatorname{dist}(K,\zeta)$.
\end{lemma}

\begin{proof}
Since $u\in L^\infty(\Omega)$, the right-hand side $\lambda|u|^{p-2}u$ is
bounded, and standard interior $C^{1,\beta}$ regularity for degenerate and
singular quasilinear equations gives~(i); see
\cite{DiBenedetto83,Tolksdorf84,Lewis83}.

For~(ii), let $x_0\in K\cap\partial\Omega$. By the geometric hypothesis
\eqref{eq:geometric-hypothesis} and $x_0\notin\zeta$ we have
$x_0\in\Gamma_D\cup\gamma$, and since $\Gamma_D$ and $\gamma$ are relatively
open and disjoint there is $r_{x_0}>0$ such that
$B_{r_{x_0}}(x_0)\cap\partial\Omega$ is contained either in $\Gamma_D$ or in
$\gamma$. Hence, near $x_0$, the boundary condition is either purely
Dirichlet or purely Robin. On a Dirichlet portion, standard boundary
$C^{1,\beta}$ regularity for the homogeneous Dirichlet problem applies; see
\cite{Tolksdorf84,Lieberman88}. On a Robin portion, the condition
\[
|\nabla u|^{p-2}\nabla u\cdot\nu+h|u|^{p-2}u=0
\]
is treated as a nonlinear conormal boundary condition; since
$h\in C^1(\overline\gamma)$ and $u$ is bounded, the $p$-growth, continuity and
monotonicity hypotheses of the boundary regularity theory are satisfied, see
\cite{Lieberman88}, and $u$ is locally $C^{1,\beta}$ up to that portion.

A finite covering of $K\cap\partial\Omega$, together with the interior
estimate on the part of $K$ away from the boundary, proves~(ii), and~(iii)
follows immediately. If $\zeta=\emptyset$, compactness of $\partial\Omega$
yields a finite boundary covering by patches of a single type, and combining
the corresponding estimates with~(i) proves~(iv).
\end{proof}

The preceding regularity allows the weak boundary conditions to be
interpreted pointwise. Note that $\Gamma_D\cap\zeta=\emptyset$: a point of
the relatively open set $\Gamma_D$ has a boundary neighborhood disjoint from
$\gamma$ and therefore cannot lie in $\overline\gamma$. The same argument
gives $\gamma\cap\zeta=\emptyset$, so every point of $\Gamma_D$ and of
$\gamma$ lies in a compact boundary portion away from $\zeta$.

\begin{remark}\label{rem:pointwise-boundary}
Let $u$ be as in Lemma~\ref{lem:regularity}. Then $u=0$ pointwise on
$\Gamma_D$, and
\[
|\nabla u|^{p-2}\nabla u\cdot\nu+h|u|^{p-2}u=0
\qquad\text{pointwise on }\gamma .
\]
Indeed, $\operatorname{div}F(\nabla u)=-\lambda|u|^{p-2}u$ in the
distributional sense, so the Gauss--Green formula identifies the weak
conormal trace with $F(\nabla u)\cdot\nu$, and localizing
\eqref{eq:weak_formulation} near $\gamma$ gives the stated identity in the
weak boundary sense. By Lemma~\ref{lem:regularity}(ii) both $u$ and
$F(\nabla u)$ are continuous up to compact boundary portions away from
$\zeta$, which upgrades the identity to a pointwise one.
\end{remark}

Where the gradient does not vanish, the degeneracy of the $p$-Laplacian
disappears and the preceding regularity improves. We record only the range
$p\ge2$, which is the one needed in the later linearized and uniqueness
arguments.
\begin{lemma}\label{lem:nondegenerate-regularity}
Let $p\ge2$, let $u$ be a positive solution as in
Lemma~\ref{lem:regularity}, and let $\beta$ be as in that lemma.

\begin{enumerate}[(i)]
\item If $\mathcal O\subset\Omega$ is open and
$\inf_{\mathcal O}|\nabla u|>0$, then there exists $\alpha\in(0,\beta]$ such
that
\[
u\in C^{2,\alpha}_{\mathrm{loc}}(\mathcal O);
\]
in particular $\nabla u$ is locally Lipschitz in $\mathcal O$.

\item Let $\mathcal O\subset\mathbb R^n$ be open with
$\mathcal O\cap\partial\Omega\subset\Gamma_D$, and suppose that
$\mathcal O\cap\partial\Omega$ is of class $C^{2,\alpha_0}$ for some
$\alpha_0\in(0,1)$. If $\inf_{\mathcal O\cap\Omega}|\nabla u|>0$, then for
every $\mathcal O'\Subset\mathcal O$ there exists
$\alpha\in\bigl(0,\min\{\beta,\alpha_0\}\bigr]$ such that
\[
u\in C^{2,\alpha}\bigl(\overline{\mathcal O'\cap\Omega}\bigr).
\]
\end{enumerate}
\end{lemma}

\begin{proof}
Let $\mathcal U\Subset\mathcal O$. By Lemma~\ref{lem:regularity}, $\nabla u$
is bounded and H\"older continuous on $\mathcal U$, so that
$0<m\le|\nabla u|\le M<\infty$ there, and by \eqref{eq:DF-two-sided} the
matrix $DF(\nabla u)$ is symmetric and uniformly elliptic on $\mathcal U$;
moreover $F$ is smooth on the annulus $\{m/2<|z|<2M\}$.

For~(i), the equation reads
$DF(\nabla u):D^2u=-\lambda u^{p-1}$ on $\mathcal U$, with coefficients
$DF(\nabla u)$ and right-hand side both H\"older continuous of exponent
$\beta$, and with uniformly elliptic principal part. Interior Schauder theory
therefore gives $u\in C^{2,\alpha}(\mathcal U)$ for some $\alpha\in(0,\beta]$;
see \cite{GT}. Since $\mathcal U\Subset\mathcal O$ was arbitrary, this proves
the local statement, and $\nabla u\in C^{1,\alpha}(\mathcal U)$ is in
particular Lipschitz on $\mathcal U$.

For~(ii), Lemma~\ref{lem:regularity} gives $u\in C^{1,\beta}$ locally up to
the Dirichlet boundary, and since $p\ge2$ the function $u^{p-1}$ is H\"older
continuous up to the boundary. The uniform lower bound on $|\nabla u|$
removes the degeneracy of the principal part. After flattening the
$C^{2,\alpha_0}$ boundary, boundary Schauder estimates for uniformly elliptic
quasilinear equations with homogeneous Dirichlet data give
$u\in C^{2,\alpha}(\overline{\mathcal O'\cap\Omega})$ with
$\alpha\le\min\{\beta,\alpha_0\}$, the restriction $\alpha\le\beta$ coming
from the H\"older exponent of the frozen coefficients $DF(\nabla u)$; see for
example \cite{GT,Lieberman88}.
\end{proof}
Only the $C^{2,\alpha}$ regularity in
Lemma~\ref{lem:nondegenerate-regularity}(ii), and hence the local boundedness
of $D^2u$, is used later, in the regularity analysis of the path coefficient
in Section~\ref{sec:linearized}. Part~(ii) requires
$\mathcal O\cap\partial\Omega$ to be of class $C^{2,\alpha_0}$, which is
slightly stronger than the standing hypothesis $\partial\Omega\in C^2$ of
Section~\ref{sec:problem-setting}; this additional local regularity of the
measurement boundary is therefore part of the hypotheses of the uniqueness
theorem for $p>2$, and is recalled where it is used.

\subsection{Unique continuation: a local boundary statement and a
propagation statement}\label{sec:ucp}

The first ingredient states that vanishing Cauchy data on a boundary patch
force the solution to vanish there.

\begin{theorem}\label{thm:boundary_ucp_local}
Let $\Omega\subset\mathbb R^n$, $n\ge2$, be a bounded domain, and let
$U\subset\mathbb R^n$ be open and connected, with
$U\cap\partial\Omega\neq\emptyset$ and $\partial\Omega\cap U$ of class $C^2$.
Let $A\in C^{0,1}(\overline{U\cap\Omega};\mathbb R^{n\times n})$ be symmetric
and uniformly elliptic on $U\cap\Omega$, say
$\theta I\le A\le\Theta I$ with $0<\theta\le\Theta<\infty$, let
$B\in L^\infty(U\cap\Omega)$ and $\lambda\in\mathbb R$, and let
$w\in H^1(U\cap\Omega)$ solve
$-\operatorname{div}(A\nabla w)=\lambda Bw$ in $U\cap\Omega$ with vanishing
Cauchy data,
\begin{equation}\label{eq:ucp-cauchy}
w=0
\quad\text{and}\quad
A\nabla w\cdot\nu=0
\qquad\text{on }U\cap\partial\Omega,
\end{equation}
the first in the sense of traces and the second in the weak conormal-trace
sense. Equivalently,
\begin{equation}\label{eq:ucp-local-weak}
\int_{U\cap\Omega}A\nabla w\cdot\nabla\varphi\,dx
=\lambda\int_{U\cap\Omega}B\,w\,\varphi\,dx
\qquad\text{for all }\varphi\in C_c^\infty(U),
\end{equation}
together with $w=0$ on $U\cap\partial\Omega$. Then $w\equiv0$ in
$U\cap\Omega$.
\end{theorem}

\begin{proof}
Formulation \eqref{eq:ucp-local-weak} encodes both the interior equation and
the vanishing of the conormal trace, since the test functions are not
required to vanish on $U\cap\partial\Omega$.

Let $K:=\{S\in\mathbb R^{n\times n}_{\mathrm{sym}}:\theta I\le S\le\Theta I\}$,
a closed convex subset of the space of symmetric matrices, and let
$\Pi$ denote the metric projection onto $K$ in the Frobenius inner product,
which is $1$-Lipschitz. Extend each entry of $A$ from
$\overline{U\cap\Omega}$ to a Lipschitz function on $U$ and set
$\widetilde A:=\Pi(A_{\mathrm{ext}})$. Then $\widetilde A$ is Lipschitz,
symmetric and uniformly elliptic on $U$, and $\widetilde A=A$ on
$U\cap\Omega$ because $A(x)\in K$ there. Extend $B$ by zero and let
$\widetilde w$ be the extension of $w$ by zero to $U$.

Since $\partial\Omega\cap U$ is $C^2$ and $w$ has vanishing Dirichlet trace
there, $\widetilde w\in H^1(U)$ with $\nabla\widetilde w=\nabla w\,\mathbf
1_\Omega$. For $\varphi\in C_c^\infty(U)$, \eqref{eq:ucp-local-weak} gives
\[
\int_U \widetilde A\,\nabla\widetilde w\cdot\nabla\varphi\,dx
=\int_{U\cap\Omega} A\,\nabla w\cdot\nabla\varphi\,dx
=\lambda\int_U \widetilde B\,\widetilde w\,\varphi\,dx ,
\]
so no distribution supported on the interface is produced and $\widetilde w$
solves the equation weakly in $U$. Because $\partial\Omega\cap U$ is $C^2$
and nonempty, the open set $U\setminus\overline\Omega$ is nonempty, and
$\widetilde w$ vanishes on it. As $U$ is connected, interior unique
continuation for uniformly elliptic divergence-form operators with Lipschitz
principal coefficients \cite{Koch01,Is} gives $\widetilde w\equiv0$ in $U$,
hence $w\equiv0$ in $U\cap\Omega$. See also
\cite{AdolfssonEscauriaza1997} for related boundary unique-continuation
results.
\end{proof}

The second ingredient propagates the vanishing through a connected region on
which the coefficients are nondegenerate.

\begin{theorem}\label{thm:interior_propagation}
Let $D\subset\mathbb R^n$ be open and connected, let $A$ be symmetric,
locally uniformly elliptic and locally Lipschitz in $D$, and let
$B\in L^\infty_{\mathrm{loc}}(D)$, $\lambda\in\mathbb R$. If
$w\in H^1_{\mathrm{loc}}(D)$ is a weak solution of
$-\operatorname{div}(A\nabla w)=\lambda Bw$ in $D$ and $w$ vanishes on some
nonempty open subset of $D$, then $w\equiv0$ in $D$.
\end{theorem}

\begin{proof}
Let $Z$ be the set of points of $D$ having a neighborhood on which $w$
vanishes; $Z$ is open and nonempty. Let $x_0$ lie in the closure of $Z$ in
$D$ and choose $B_r(x_0)\Subset D$. There $A$ is uniformly elliptic and
Lipschitz and $B$ is bounded, so the strong unique continuation property
holds \cite{Koch01}; since $Z$ is open, $B_r(x_0)\cap Z$ is a nonempty open
set on which $w$ vanishes, whence $w\equiv0$ on $B_r(x_0)$ and $x_0\in Z$.
Thus $Z$ is open and closed in the connected set $D$, so $Z=D$.
\end{proof}

The next remark explains why the Lipschitz regularity of the leading
coefficient cannot simply be dropped in dimensions three and higher.

\begin{remark}\label{rem:ucp-regularity}
Sub-Lipschitz unique continuation can fail in dimensions $n\ge3$: Miller
\cite{Miller74} constructed a self-adjoint divergence-form counterexample
with H\"older continuous principal coefficients, following earlier
nonuniqueness phenomena for second-order elliptic equations due to Pli\'s
\cite{Plis63}. This is directly relevant to the $p$-Laplace linearization,
whose coefficients lose regularity at critical points of the reference
eigenfunction.
\end{remark}

In the plane the situation is better, as the following remark records.
\begin{remark}\label{rem:ucp-2d}\label{rem:propagation-2d}
In dimension $n=2$, unique continuation for divergence-form elliptic equations
holds with $A$ merely measurable, symmetric and uniformly elliptic, by the
Bers--Nirenberg representation and quasiconformal methods; see
Alessandrini~\cite{Alessandrini12} and Schulz~\cite{Schulz98}. The two proofs
above use only the ellipticity of the extension, the vanishing of the Cauchy data,
and the unique continuation property on balls, and the projection $\Pi$
preserves measurability; both theorems   remain valid for $n=2$ under
that weaker hypothesis on $A$.
\end{remark}

For completeness, we state the form of the implicit function theorem used in
Section~\ref{sec:frechet}, together with the modulus of continuity of the
remainder that it supplies.

\begin{theorem}\label{thm:implicit_function}
Let $X,Y,Z$ be Banach spaces, $\mathcal O\subset X\times Y$ open, and
$\mathcal G:\mathcal O\to Z$ continuously Fr\'echet differentiable. Suppose
$\mathcal G(x_0,y_0)=0$ and $D_y\mathcal G(x_0,y_0):Y\to Z$ is a bounded
linear isomorphism. Then there are neighborhoods $U\ni x_0$ and $V\ni y_0$
and a unique $C^1$ map $\Upsilon:U\to V$ with $\Upsilon(x_0)=y_0$ and
$\mathcal G(x,\Upsilon(x))=0$ on $U$. After shrinking $U$ there is a
nondecreasing modulus $\omega$ with $\omega(s)\to0$ as $s\to0$ such that
\begin{equation}\label{eq:IFT-modulus}
\|\Upsilon(x)-\Upsilon(x_0)-D\Upsilon(x_0)[x-x_0]\|_Y
\le\omega(\|x-x_0\|_X)\,\|x-x_0\|_X .
\end{equation}
\end{theorem}

For the statement see \cite{Zeidler86}; the modulus is obtained by taking
$\omega(s)$ to be the supremum of the quotient on the left-hand side of
\eqref{eq:IFT-modulus} divided by $\|x-x_0\|_X$ over
$0<\|x-x_0\|_X\le s$, which is nondecreasing and tends to zero as
$s\downarrow0$ by the differentiability of $\Upsilon$ at $x_0$.
\section{Thin-Coating Asymptotics and Effective Robin Conditions}
\label{sec:coating}

Thin layers with material properties different from those of the underlying
medium generate an effective boundary interaction when their thickness tends
to zero. For clarity, the asymptotic theorem is proved when the entire
boundary $\Gamma:=\partial\Omega$ is coated, so that no junction effects occur
at the edge of the coating; this is the one place where the mixed geometry of
Section~\ref{sec:problem-setting} is replaced by $\Gamma_D=\emptyset$. The
partially coated case, which is the configuration underlying the inverse
problem of Section~\ref{Inv}, is described in Remark~\ref{rem:mixed-coating}.

Let $\Omega$ be a bounded domain with $C^2$ boundary $\Gamma$, and let
$\rho\in C^2(\Gamma)$ satisfy $\rho>0$. Denote by $\nu_\Gamma$ the outward
unit normal and by
\begin{equation}\label{eq:weingarten}
W(\xi):=-D_\Gamma\nu_\Gamma(\xi)
\end{equation}
the Weingarten map (shape operator) of $\Gamma$, where
$D_\Gamma\nu_\Gamma$ denotes the tangential derivative of the unit normal.
With this sign convention and the outward normal, a ball of radius $R$ has
$W=-R^{-1}I$, so that the Jacobian \eqref{eq:jacobian} below equals
$(1+t/R)^{n-1}$, as it must for an outward layer. Only $\rho$, its tangential
gradient and $\inf_\Gamma\rho>0$ enter the estimates, so $\rho\in C^1(\Gamma)$
with $\rho>0$ would suffice throughout this section.

For $\varepsilon>0$ sufficiently small, the tubular map
\[
X(\xi,t):=\xi+t\nu_\Gamma(\xi)
\]
is a diffeomorphism onto its image for
$0\le t\le\varepsilon\rho(\xi)$. We define
\begin{equation}\label{eq:thin_layer}
\Sigma_\varepsilon
:=
\bigl\{
X(\xi,t):
\xi\in\Gamma,\ 0<t<\varepsilon\rho(\xi)
\bigr\},
\qquad
\Omega_\varepsilon
:=
\Omega\cup\Gamma\cup\Sigma_\varepsilon .
\end{equation}
The conductivity is
\begin{equation}\label{eq:sigma-eps}
\sigma_\varepsilon(x)=
\begin{cases}
1, & x\in\Omega,\\[1mm]
\varepsilon^{p-1}, & x\in\Sigma_\varepsilon ,
\end{cases}
\end{equation}
see Figure~\ref{fig:p_laplace_thin_coating}. This is the critical thin-layer
scaling: a normal variation across a layer of thickness
$\varepsilon\rho(\xi)$ is of order $u/(\varepsilon\rho(\xi))$, so the
corresponding nonlinear flux has the scale
\begin{equation}\label{eq:scaling-heuristic}
\varepsilon^{p-1}
\Bigl|\frac{u}{\varepsilon\rho(\xi)}\Bigr|^{p-2}
\frac{u}{\varepsilon\rho(\xi)}
=\frac{|u|^{p-2}u}{\rho(\xi)^{p-1}},
\end{equation}
which is of order one as $\varepsilon\to0$; that this is the only order with a
nondegenerate limit is the content of Remark~\ref{rem:criticality}.

The choice \eqref{eq:sigma-eps} specifies a family of increasingly thin and
increasingly resistive layers, not the intrinsic conductivity of a fixed
material. For $p=2$ it reduces to $\sigma_\varepsilon=\varepsilon$, that is to
Friedman's critical scaling $\sigma_\varepsilon/\varepsilon\to\alpha$ with
$\alpha=1$, and the effective condition $\rho\,\partial_\nu u+u=0$ is the
classical contact-resistance law. For $p\neq2$ the constitutive law
$J=\sigma|E|^{p-2}E$ with $E=-\nabla u$ carries the same $(p-1)$-homogeneity
as the Ostwald--de Waele power-law model, see for example \cite{bird1987}, and
\eqref{eq:sigma-eps} is the order that balances the vanishing thickness
against that homogeneity.

\begin{figure}[ht!]
\centering
\begin{tikzpicture}[scale=0.85]

\def\angXi{35}
\def\angXiP{300}
\pgfmathsetmacro{\rB}{2.45+0.16*cos(3*\angXiP)}

\fill[gray!28]
plot[domain=0:360, samples=200, variable=\t]
({(2.45+0.16*cos(3*\t))*cos(\t)},{(2.45+0.16*cos(3*\t))*sin(\t)}) -- cycle;

\fill[blue!8] (0,0) circle (2);
\draw[blue!60!black, very thick] (0,0) circle (2);
\draw[gray!75, thick, dashed]
plot[domain=0:360, samples=200, variable=\t]
({(2.45+0.16*cos(3*\t))*cos(\t)},{(2.45+0.16*cos(3*\t))*sin(\t)});

\coordinate (Xi) at ({2*cos(\angXi)},{2*sin(\angXi)});
\fill (Xi) circle (1.6pt);
\node[font=\scriptsize, below right] at (Xi) {$\xi$};
\draw[->, green!50!black, thick] (Xi) -- ({3.15*cos(\angXi)},{3.15*sin(\angXi)});
\node[green!50!black, font=\scriptsize]
at ({3.62*cos(\angXi)},{3.62*sin(\angXi)}) {$\nu_\Gamma(\xi)$};

\coordinate (XiP)    at ({2*cos(\angXiP)},{2*sin(\angXiP)});
\coordinate (XiPOut) at ({\rB*cos(\angXiP)},{\rB*sin(\angXiP)});
\fill (XiP) circle (1.6pt);
\node[font=\scriptsize, above left] at (XiP) {$\xi'$};
\draw[<->, red!65!black, thick] (XiP) -- (XiPOut);
\coordinate (ThickLab) at ({3.30*cos(\angXiP)},{3.30*sin(\angXiP)});
\node[red!65!black, font=\scriptsize] at (ThickLab) {$\varepsilon\rho(\xi')$};
\draw[red!65!black, thin] (ThickLab) -- ($ (XiP)!0.5!(XiPOut) $);

\node at (0,0.32) {$\Omega$};
\node[font=\scriptsize] at (0,-0.14) {$\sigma_\varepsilon=1$};

\node[font=\scriptsize] at (0,-3.55)
{$\Sigma_\varepsilon,\quad \sigma_\varepsilon=\varepsilon^{p-1}$};
\draw[->, thin] (0,-3.35) -- (0,-2.22);

\node[blue!60!black, font=\scriptsize] at (-3.55,1.35) {$\Gamma=\partial\Omega$};
\draw[->, thin, blue!60!black] (-3.05,1.25) -- ({2*cos(152)},{2*sin(152)});

\node[gray!55!black, font=\scriptsize] at (3.75,-1.75)
{$\partial\Omega_\varepsilon,\ \ \Phi_\varepsilon=0$};
\draw[->, thin, gray!55!black] (3.05,-1.62) -- ({2.45*cos(-30)},{2.45*sin(-30)});

\end{tikzpicture}
\caption{Schematic thin-coating geometry. The domain $\Omega$ is surrounded by
a layer $\Sigma_\varepsilon$ of variable thickness $\varepsilon\rho$. The
conductivity equals $1$ in $\Omega$ and $\varepsilon^{p-1}$ in the coating.
The dashed curve is the outer boundary, where a homogeneous Dirichlet
condition is imposed. The normal direction is drawn at $\xi$ and the layer
thickness at a second point $\xi'$.}
\label{fig:p_laplace_thin_coating}
\end{figure}

We consider the two-phase principal eigenvalue problem
\begin{equation}\label{eq:two-phase}
\begin{cases}
-\operatorname{div}\bigl(\sigma_\varepsilon|\nabla\Phi_\varepsilon|^{p-2}
\nabla\Phi_\varepsilon\bigr)
=\Lambda_1(\varepsilon)|\Phi_\varepsilon|^{p-2}\Phi_\varepsilon
& \text{in }\Omega_\varepsilon,\\[2mm]
\Phi_\varepsilon=0 & \text{on }\partial\Omega_\varepsilon .
\end{cases}
\end{equation}
Problem \eqref{eq:two-phase} is a purely Dirichlet eigenvalue problem on
$\Omega_\varepsilon$ for the weighted operator with coefficient
$\sigma_\varepsilon$, which for each fixed $\varepsilon>0$ is bounded above by
$1$ and below by $\varepsilon^{p-1}$. By the last part of
Remark~\ref{rem:pure-robin}, the principal eigenvalue
$\Lambda_1(\varepsilon)$ therefore exists and is simple, and
$\Phi_\varepsilon$ may be chosen positive with
$\|\Phi_\varepsilon\|_{L^p(\Omega_\varepsilon)}=1$.  Note that the degeneracy as $\varepsilon\to0$ is not an issue here: only fixed
$\varepsilon$ is involved, and the asymptotics of
Theorem~\ref{thm:asymptotic} are extracted from the Rayleigh quotients, not
from these bounds.

 The interface $\Gamma$ is
internal to $\Omega_\varepsilon$. Since $\Phi_\varepsilon\in W^{1,p}(\Omega_\varepsilon)$, its traces from the
two sides of $\Gamma$ agree, while the weak formulation of
\eqref{eq:two-phase} yields continuity of the weighted normal flux
$\sigma_\varepsilon|\nabla\Phi_\varepsilon|^{p-2}
\nabla\Phi_\varepsilon\cdot\nu_\Gamma$
across $\Gamma$ in the usual weak transmission sense.

The next lemma collects everything the layer contributes to the asymptotics.
Parts~(i) and~(ii) compute the mass and the energy carried by the affine
extension of a fixed function and identify the surface term that survives in
the limit. Parts~(iii) and~(iv) are one-dimensional estimates along the normal
fibres, valid for an arbitrary competitor: the first says that the layer
carries negligible $L^p$ mass, the second bounds the layer energy from below
by the effective Robin energy of the trace.

\begin{lemma}\label{lem:layer}
Let $\varepsilon>0$ be small enough that $X$ is a diffeomorphism on the layer.
\begin{enumerate}[(i)]
\item
Let $u\in C^1(\overline\Omega)$ and define $\widetilde u_\varepsilon$ on
$\Omega_\varepsilon$ by
\begin{equation}\label{eq:affine-extension}
\widetilde u_\varepsilon(x)
=
\begin{cases}
u(x), & x\in\Omega,\\[1mm]
u(\xi)\displaystyle
\left(1-\frac{t}{\varepsilon\rho(\xi)}\right),
& x=X(\xi,t)\in\Sigma_\varepsilon .
\end{cases}
\end{equation}
Then $\widetilde u_\varepsilon\in W_0^{1,p}(\Omega_\varepsilon)$ and
\begin{equation}\label{eq:layer-Lp-affine}
\int_{\Sigma_\varepsilon}
|\widetilde u_\varepsilon|^p\,dx
=O(\varepsilon).
\end{equation}

\item
For the same extension,
\begin{equation}\label{eq:layer-energy-affine}
\int_{\Sigma_\varepsilon}
\varepsilon^{p-1}
|\nabla\widetilde u_\varepsilon|^p\,dx
=
\int_\Gamma
\frac{|u(\xi)|^p}{\rho(\xi)^{p-1}}\,d\sigma
+O(\varepsilon).
\end{equation}

\item
For every $\Phi\in W^{1,p}_0(\Omega_\varepsilon)$,
\begin{equation}\label{eq:Sigma_Lp_small}
\int_{\Sigma_\varepsilon}|\Phi|^p\,dx
\le
C\varepsilon^p\int_{\Sigma_\varepsilon}|\nabla\Phi|^p\,dx .
\end{equation}

\item
For every $\Phi\in W^{1,p}_0(\Omega_\varepsilon)$, with $T\Phi$ the trace of
$\Phi|_\Omega$ on $\Gamma$,
\begin{equation}\label{eq:slice-lower-bound}
\varepsilon^{p-1}\int_{\Sigma_\varepsilon}|\nabla\Phi|^p\,dx
\ \ge\
(1-C\varepsilon)\int_\Gamma
\frac{|T\Phi(\xi)|^p}{\rho(\xi)^{p-1}}\,d\sigma .
\end{equation}
\end{enumerate}
The constants in \eqref{eq:Sigma_Lp_small} and \eqref{eq:slice-lower-bound}
depend only on $p$, $\Gamma$ and $\|\rho\|_{C^1(\Gamma)}$, while those implicit
in \eqref{eq:layer-Lp-affine} and \eqref{eq:layer-energy-affine} depend in
addition on $\|u\|_{C^1(\overline\Omega)}$ and on
$\bigl(\inf_\Gamma\rho\bigr)^{-1}$.
\end{lemma}

\begin{proof}
In tubular coordinates,
\begin{equation}\label{eq:jacobian}
dx=J(\xi,t)\,dt\,d\sigma(\xi),
\qquad
J(\xi,t)=\det(I-tW(\xi))=1+O(t)=1+O(\varepsilon)
\end{equation}
uniformly for $0\le t\le\varepsilon\rho(\xi)$; in particular
$1-C\varepsilon\le J\le1+C\varepsilon$ and
$|\Sigma_\varepsilon|=O(\varepsilon)$.

\emph{(i).} Write
\[
\widehat u_\varepsilon(\xi,t)
:=
u(\xi)\left(1-\frac{t}{\varepsilon\rho(\xi)}\right),
\qquad
0\le t\le\varepsilon\rho(\xi),
\]
so that $\widehat u_\varepsilon(\xi,0)=u(\xi)$ and
$\widehat u_\varepsilon(\xi,\varepsilon\rho(\xi))=0$. Hence the interior
function and the layer extension have the same trace on $\Gamma$, while the
extension has zero trace on $\partial\Omega_\varepsilon$; since the tubular
coordinate map is bi-Lipschitz for all sufficiently small $\varepsilon$, the
piecewise-defined function \eqref{eq:affine-extension} belongs to
$W_0^{1,p}(\Omega_\varepsilon)$. As
$|\widehat u_\varepsilon|\le\|u\|_{L^\infty(\Gamma)}$ and
$|\Sigma_\varepsilon|=O(\varepsilon)$, \eqref{eq:layer-Lp-affine} follows.

\emph{(ii).} At fixed $\xi$,
\[
\partial_t\widehat u_\varepsilon(\xi,t)
=
-\frac{u(\xi)}{\varepsilon\rho(\xi)},
\]
and at fixed $t$,
\[
\nabla_\Gamma\widehat u_\varepsilon(\xi,t)
=
\left(1-\frac{t}{\varepsilon\rho(\xi)}\right)
\nabla_\Gamma u(\xi)
+
\frac{t\,u(\xi)}{\varepsilon\rho(\xi)^2}
\nabla_\Gamma\rho(\xi).
\]
Because $0\le t\le\varepsilon\rho(\xi)$ and $u,\rho\in C^1(\Gamma)$ with
$\rho$ bounded away from zero,
\begin{equation}\label{eq:tangential-bound}
\sup_{\xi\in\Gamma,\,
0\le t\le\varepsilon\rho(\xi)}
|\nabla_\Gamma\widehat u_\varepsilon(\xi,t)|
\le C_0,
\end{equation}
with $C_0$ independent of $\varepsilon$. The Euclidean gradient in tubular
coordinates is
\begin{equation}\label{eq:tubular-metric}
|\nabla\widetilde u_\varepsilon(X(\xi,t))|^2
=
|\partial_t\widehat u_\varepsilon(\xi,t)|^2
+
\left|
(I-tW(\xi))^{-1}
\nabla_\Gamma\widehat u_\varepsilon(\xi,t)
\right|^2 ,
\end{equation}
and $(I-tW(\xi))^{-1}=I+O(\varepsilon)$ uniformly in the layer. Set
\[
N_\varepsilon(\xi):=\frac{|u(\xi)|}{\varepsilon\rho(\xi)},
\qquad
T_\varepsilon(\xi,t)
:=
\left|
(I-tW(\xi))^{-1}
\nabla_\Gamma\widehat u_\varepsilon(\xi,t)
\right| ,
\]
so that $T_\varepsilon\le C$ uniformly by \eqref{eq:tangential-bound} and
$|\nabla\widetilde u_\varepsilon|^p
=(N_\varepsilon^2+T_\varepsilon^2)^{p/2}$ by \eqref{eq:tubular-metric}. For
$p>1$ and $N,T\ge0$,
\[
0
\le
\bigl(N^2+T^2\bigr)^{p/2}-N^p
\le
C_p\bigl(N^{p-1}T+T^p\bigr),
\]
whence
\[
|\nabla\widetilde u_\varepsilon|^p
=
\left|\frac{u(\xi)}{\varepsilon\rho(\xi)}\right|^p
+R_\varepsilon(\xi,t),
\qquad
0\le R_\varepsilon(\xi,t)
\le
C\left[\left|\frac{u(\xi)}{\varepsilon\rho(\xi)}\right|^{p-1}+1\right].
\]
Using \eqref{eq:jacobian}, the principal term satisfies
\[
\varepsilon^{p-1}
\int_\Gamma\int_0^{\varepsilon\rho(\xi)}
\left|\frac{u(\xi)}{\varepsilon\rho(\xi)}\right|^p
J(\xi,t)\,dt\,d\sigma(\xi)
=
\int_\Gamma\frac{|u(\xi)|^p}{\rho(\xi)^{p-1}}\,d\sigma(\xi)
+O(\varepsilon),
\]
while for the remainder
\begin{align*}
\varepsilon^{p-1}
\int_\Gamma\int_0^{\varepsilon\rho(\xi)}
R_\varepsilon(\xi,t)J(\xi,t)\,dt\,d\sigma(\xi)
&\le
C\varepsilon^{p-1}
\int_\Gamma\int_0^{\varepsilon\rho(\xi)}
\left[\frac{|u(\xi)|^{p-1}}{\varepsilon^{p-1}\rho(\xi)^{p-1}}+1\right]
dt\,d\sigma(\xi)
\\
&=O(\varepsilon).
\end{align*}
This proves \eqref{eq:layer-energy-affine}.

\emph{(iii) and (iv).} By the Sobolev slicing property, for a.e.\
$\xi\in\Gamma$ the function $g_\xi(t):=\Phi(X(\xi,t))$ belongs to
$W^{1,p}(0,\varepsilon\rho(\xi))$ with
$|g_\xi'(t)|\le|\nabla\Phi(X(\xi,t))|$, and
$g_\xi(\varepsilon\rho(\xi))=0$ because $\Phi$ vanishes on
$\partial\Omega_\varepsilon$. H\"older's inequality applied to
$g_\xi(t)=-\int_t^{\varepsilon\rho(\xi)}g_\xi'(s)\,ds$ gives
\[
|g_\xi(t)|^p
\le
(\varepsilon\rho(\xi))^{p-1}
\int_0^{\varepsilon\rho(\xi)}|g_\xi'|^p\,ds ,
\qquad
0\le t\le\varepsilon\rho(\xi).
\]
Integrating in $t$ yields
\[
\int_0^{\varepsilon\rho(\xi)}|g_\xi|^p\,dt
\le
(\varepsilon\rho(\xi))^{p}
\int_0^{\varepsilon\rho(\xi)}|g_\xi'|^p\,dt ,
\]
and integrating over $\Gamma$ against \eqref{eq:jacobian} proves
\eqref{eq:Sigma_Lp_small}. Taking instead $t=0$ in the same H\"older bound
gives
\[
\varepsilon^{p-1}\int_0^{\varepsilon\rho(\xi)}|g_\xi'|^p\,dt
\ \ge\
\frac{|g_\xi(0)|^p}{\rho(\xi)^{p-1}}
=
\frac{|T\Phi(\xi)|^p}{\rho(\xi)^{p-1}},
\]
the powers of $\varepsilon$ cancelling exactly; integrating over $\Gamma$ and
using $J\ge1-C\varepsilon$ proves \eqref{eq:slice-lower-bound}.
\end{proof}

The following theorem is the thin-coating limit: the principal eigenvalue of
the two-phase problem converges to that of the effective Robin problem, and
the eigenfunctions converge strongly.

\begin{theorem}\label{thm:asymptotic}
Let $\Lambda_1(\varepsilon)$ and $\Phi_\varepsilon$ be as above, and let
$(\mu_1,u_1)$ be the normalized positive principal eigenpair of
\begin{equation}\label{eq:robin_limit1}
\begin{cases}
-\Delta_pu=\mu|u|^{p-2}u & \text{in }\Omega,\\[2mm]
|\nabla u|^{p-2}\partial_{\nu_\Gamma}u+\rho^{-(p-1)}|u|^{p-2}u=0
& \text{on }\Gamma .
\end{cases}
\end{equation}
Then $\Lambda_1(\varepsilon)\to\mu_1$ and $\Phi_\varepsilon|_\Omega\to u_1$
strongly in $W^{1,p}(\Omega)$ as $\varepsilon\to0$.
\end{theorem}

\begin{proof}
Existence, positivity and simplicity of $(\mu_1,u_1)$ follow from
Remark~\ref{rem:pure-robin}, whose hypothesis $\inf_\Gamma\rho^{-(p-1)}>0$
holds here. The Rayleigh characterizations are
\begin{equation}\label{eq:rayleigh_eps}
\begin{aligned}
\Lambda_1(\varepsilon)
&=
\inf_{\Phi\in W_0^{1,p}(\Omega_\varepsilon)\setminus\{0\}}
\frac{\displaystyle
\int_{\Omega_\varepsilon}\sigma_\varepsilon|\nabla\Phi|^p\,dx}
{\displaystyle
\int_{\Omega_\varepsilon}|\Phi|^p\,dx},
\\[1mm]
\mu_1
&=
\inf_{u\in W^{1,p}(\Omega)\setminus\{0\}}
\frac{\displaystyle
\int_\Omega|\nabla u|^p\,dx
+\int_\Gamma\rho^{-(p-1)}|u|^p\,d\sigma}
{\displaystyle
\int_\Omega|u|^p\,dx}.
\end{aligned}
\end{equation}

\medskip\noindent\textbf{Step 1: upper bound.}
Choose $u^{(m)}\in C^1(\overline\Omega)$ with $u^{(m)}\to u_1$ strongly in
$W^{1,p}(\Omega)$. By continuity of the trace operator, the limiting Robin
Rayleigh quotient of $u^{(m)}$ converges to $\mu_1$ as $m\to\infty$. Fix $m$
and extend $u^{(m)}$ affinely through the layer as in
\eqref{eq:affine-extension}. The constants in Lemma~\ref{lem:layer}(i)--(ii)
depend on $\|u^{(m)}\|_{C^1(\overline\Omega)}$ and are therefore not uniform
in $m$, so the limits must be taken in the order $\varepsilon\to0$ first and
$m\to\infty$ second. With that order, Lemma~\ref{lem:layer}(i)--(ii) and
\eqref{eq:rayleigh_eps} give
\[
\limsup_{\varepsilon\to0}\Lambda_1(\varepsilon)
\le
\frac{\displaystyle
\int_\Omega|\nabla u^{(m)}|^p\,dx
+\int_\Gamma\rho^{-(p-1)}|u^{(m)}|^p\,d\sigma}
{\displaystyle
\int_\Omega|u^{(m)}|^p\,dx}
\qquad\text{for every }m,
\]
and letting $m\to\infty$,
\begin{equation}\label{eq:upper_bound}
\limsup_{\varepsilon\to0}\Lambda_1(\varepsilon)\le\mu_1.
\end{equation}

\medskip\noindent\textbf{Step 2: compactness and vanishing coating mass.}
Let $\varepsilon_k\to0$ be an arbitrary sequence and set
$\Phi_k:=\Phi_{\varepsilon_k}$ and $v_k:=\Phi_k|_\Omega$. By
\eqref{eq:upper_bound}, $\Lambda_1(\varepsilon_k)\le C$, and testing
\eqref{eq:two-phase} with $\Phi_k$, using
$\|\Phi_k\|_{L^p(\Omega_{\varepsilon_k})}=1$, gives
\begin{equation}\label{eq:energy_Phi}
\int_\Omega|\nabla v_k|^p\,dx
+\varepsilon_k^{p-1}
\int_{\Sigma_{\varepsilon_k}}|\nabla\Phi_k|^p\,dx
=
\Lambda_1(\varepsilon_k)\le C.
\end{equation}
Since $\|v_k\|_{L^p(\Omega)}\le1$, the sequence $(v_k)$ is bounded in
$W^{1,p}(\Omega)$. By \eqref{eq:Sigma_Lp_small} and \eqref{eq:energy_Phi},
\[
\int_{\Sigma_{\varepsilon_k}}|\Phi_k|^p\,dx
\le
C\varepsilon_k^p\int_{\Sigma_{\varepsilon_k}}|\nabla\Phi_k|^p\,dx
\le
C\varepsilon_k,
\]
so the normalization of $\Phi_k$ implies
$\|v_k\|_{L^p(\Omega)}^p=1+O(\varepsilon_k)$.

After passing to a subsequence, not relabeled,
\begin{align}
v_k&\rightharpoonup\widehat u
&&\text{weakly in }W^{1,p}(\Omega),
\label{eq:weak_conv}\\
v_k&\to\widehat u
&&\text{strongly in }L^p(\Omega),
\label{eq:strong_Lp}\\
Tv_k&\to T\widehat u
&&\text{strongly in }L^p(\Gamma),
\label{eq:trace_conv}
\end{align}
the last by compactness of the trace embedding. Since $\Phi_k>0$ we have
$\widehat u\ge0$, and \eqref{eq:strong_Lp} together with the preceding
normalization gives $\|\widehat u\|_{L^p(\Omega)}=1$.

\medskip\noindent\textbf{Step 3: lower bound.}
By Lemma~\ref{lem:layer}(iv) applied to $\Phi_k$ and by
\eqref{eq:trace_conv},
\begin{equation}\label{eq:coating-liminf}
\liminf_{k\to\infty}\varepsilon_k^{p-1}
\int_{\Sigma_{\varepsilon_k}}|\nabla\Phi_k|^p\,dx
\ \ge\ \int_\Gamma\frac{|\widehat u|^p}{\rho^{p-1}}\,d\sigma .
\end{equation}
Combining this with the weak lower semicontinuity of the Dirichlet energy,
the superadditivity of the lower limit, \eqref{eq:energy_Phi} and the
variational characterization \eqref{eq:rayleigh_eps} of $\mu_1$, together with
$\|\widehat u\|_{L^p(\Omega)}=1$, gives
\begin{equation}\label{eq:lower_bound}
\liminf_{k\to\infty}\Lambda_1(\varepsilon_k)
\ \ge\ \int_\Omega|\nabla\widehat u|^p\,dx
+\int_\Gamma\frac{|\widehat u|^p}{\rho^{p-1}}\,d\sigma\ \ge\ \mu_1 .
\end{equation}

\medskip\noindent\textbf{Step 4: identification of the limit and strong
convergence.}
Combining \eqref{eq:upper_bound} and \eqref{eq:lower_bound},
\[
\mu_1
\ge
\limsup_{k\to\infty}\Lambda_1(\varepsilon_k)
\ge
\liminf_{k\to\infty}\Lambda_1(\varepsilon_k)
\ge
\int_\Omega|\nabla\widehat u|^p\,dx
+
\int_\Gamma\frac{|\widehat u|^p}{\rho^{p-1}}\,d\sigma
\ge
\mu_1,
\]
so equality holds throughout. Hence $\Lambda_1(\varepsilon_k)\to\mu_1$ and
$\widehat u$ is a normalized nonnegative minimizer of the limiting Robin
Rayleigh quotient, so that $\widehat u=u_1$ by simplicity of the principal
eigenvalue.

To obtain strong convergence, set
\[
A_k:=\int_\Omega|\nabla v_k|^p\,dx,
\qquad
B_k:=\varepsilon_k^{p-1}
\int_{\Sigma_{\varepsilon_k}}|\nabla\Phi_k|^p\,dx,
\]
\[
A:=\int_\Omega|\nabla u_1|^p\,dx,
\qquad
B:=\int_\Gamma\frac{|u_1|^p}{\rho^{p-1}}\,d\sigma .
\]
By the energy identity, $A_k+B_k=\Lambda_1(\varepsilon_k)\to\mu_1=A+B$, while
weak lower semicontinuity and \eqref{eq:coating-liminf} give
$\liminf_kA_k\ge A$ and $\liminf_kB_k\ge B$. Therefore $A_k\to A$. Since
$\nabla v_k\rightharpoonup\nabla u_1$ in $L^p(\Omega;\mathbb R^n)$ and
$\|\nabla v_k\|_{L^p(\Omega)}\to\|\nabla u_1\|_{L^p(\Omega)}$, the uniform
convexity of $L^p$ gives $\nabla v_k\to\nabla u_1$ strongly in
$L^p(\Omega;\mathbb R^n)$, and with \eqref{eq:strong_Lp},
\[
v_k\to u_1
\qquad\text{strongly in }W^{1,p}(\Omega).
\]

The limit $(\mu_1,u_1)$ obtained in this way does not depend on the
subsequence extracted in Step 2. Consequently every sequence
$\varepsilon_k\to0$ has a subsequence along which
$\Lambda_1(\varepsilon_k)\to\mu_1$ and $\Phi_{\varepsilon_k}|_\Omega\to u_1$
in $W^{1,p}(\Omega)$, so the full sequence converges to the same limit, and
since the sequence was arbitrary,
\[
\Lambda_1(\varepsilon)\to\mu_1,
\qquad
\Phi_\varepsilon|_\Omega\to u_1
\quad\text{strongly in }W^{1,p}(\Omega)
\qquad(\varepsilon\to0).
\]
\end{proof}

Thus the thin layer generates the effective coefficient $h=\rho^{-(p-1)}$; for
$\rho\equiv1$ one gets $h\equiv1$, and for $p=2$ the formula reduces to
$h=1/\rho$, consistent with the classical linear thin-layer limit
\cite{Fr,BrezisEtAl80}.

The next remark records the partially coated geometry, which is the one
underlying the inverse problem.

\begin{remark}\label{rem:mixed-coating}
Let $\Omega$, $\Gamma_D$ and $\gamma$ satisfy the standing assumptions of
Section~\ref{sec:problem-setting}, let $\rho\in C^1(\overline\gamma)$ with
$\rho>0$ on $\overline\gamma$, and coat only the inaccessible portion:
\[
\Sigma_\varepsilon^\gamma
:=\bigl\{X(\xi,t):\xi\in\gamma,\ 0<t<\varepsilon\rho(\xi)\bigr\},
\qquad
\Omega_\varepsilon:=\Omega\cup\gamma\cup\Sigma_\varepsilon^\gamma .
\]
Impose the homogeneous Dirichlet condition on $\Gamma_D$ and on the outer face
$\{X(\xi,\varepsilon\rho(\xi)):\xi\in\gamma\}$, and the natural no-flux
condition on the lateral face $\{X(\xi,t):\xi\in\partial\gamma\}$, whose
surface measure is $O(\varepsilon)$. Then, with $\sigma_\varepsilon$ as in
\eqref{eq:sigma-eps},
\[
\Lambda_1(\varepsilon)\to\lambda_1(h),
\qquad
\Phi_\varepsilon|_\Omega\to u(h)
\quad\text{strongly in }W^{1,p}(\Omega),
\qquad
h=\rho^{-(p-1)} .
\]
The proof is that of Theorem~\ref{thm:asymptotic} with three changes. The
admissible class is
$\{\Phi\in W^{1,p}(\Omega_\varepsilon):\Phi=0$ on $\Gamma_D$ and on the outer
face$\}$, so the competitor of Step~1 is taken in $V$ and extended affinely
across $\gamma$ only, which is admissible precisely because no Dirichlet
condition is imposed on the lateral face. The slicing estimates of
Lemma~\ref{lem:layer}(iii)--(iv) are integrated over $\gamma$ instead of
$\Gamma$, the fibres over $\gamma$ terminating on the outer face where
$\Phi$ vanishes. Finally, $\lambda_1(h)>0$ and the simplicity of the limiting
eigenpair come from Proposition~\ref{prop:principal_eigenvalue} rather than
from Remark~\ref{rem:pure-robin}.
\end{remark}

The final remark shows that the exponent $p-1$ is the only order of vanishing
of the conductivity that produces a finite nonzero Robin interaction.
\begin{remark}\label{rem:criticality}
Let the conductivity in the layer be a general positive function of
$\varepsilon$, the value $1$ in $\Omega$ being unchanged:
\[
\sigma_\varepsilon=1\ \text{ in }\Omega,
\qquad
\sigma_\varepsilon=s(\varepsilon)>0\ \text{ in }\Sigma_\varepsilon,
\qquad
c_\varepsilon:=\varepsilon^{1-p}s(\varepsilon).
\]
Suppose $c_\varepsilon\to c\in(0,\infty)$. The proof of
Theorem~\ref{thm:asymptotic} then applies with $\varepsilon^{p-1}$ replaced by
$s(\varepsilon)$ throughout: parts~(ii) and~(iv) of Lemma~\ref{lem:layer}
acquire the factor $c_\varepsilon$ on their right-hand sides, and part~(iii)
gives $\int_{\Sigma_\varepsilon}|\Phi_\varepsilon|^p\,dx\le
C\varepsilon/c_\varepsilon\to0$, which is all Step~2 uses. The effective
coefficient is therefore
\[
h=c\,\rho^{-(p-1)}=\bigl(c^{-1/(p-1)}\rho\bigr)^{-(p-1)},
\]
that is, the limit is the one produced by the rescaled thickness
$c^{-1/(p-1)}\rho$. If $c_\varepsilon\to0$, the affine extension of the
constant function $1$ has Rayleigh quotient
$c_\varepsilon\int_\Gamma\rho^{-(p-1)}\,d\sigma\,/\,|\Omega|+o(1)\to0$, so
$\Lambda_1(\varepsilon)\to0$, which is the Neumann limit of perfect
insulation. The regime $c_\varepsilon\to\infty$ formally corresponds to a
Dirichlet limit and is not needed here.
\end{remark}
\section{Inverse Problem: Uniqueness and Stability}\label{Inv}

We now return to the mixed boundary geometry introduced in
Section~\ref{sec:problem-setting} and study the inverse recovery of the Robin
coefficient from spectral and boundary measurements.
Sections~\ref{sec:boundary_nondegen} and \ref{sec:eigenvalue-monotonicity} are
valid for all $1<p<\infty$; from Section~\ref{sec:linearized} onward we assume
$p\ge2$.
Recall that
\[
\zeta=\overline{\Gamma_D}\cap\overline\gamma.
\]
We fix a nonempty relatively open patch
 \begin{equation}\label{eq:measurement-patch}
\Gamma_{\mathrm m}\Subset\Gamma_D,
\qquad
d_{\mathrm m}
:=
\operatorname{dist}\bigl(
\overline{\Gamma_{\mathrm m}},
\partial\Omega\setminus\Gamma_D
\bigr)>0.
\end{equation}
in particular, \(\operatorname{dist}(\Gamma_m,\zeta)\ge d_m>0\).

For the uniqueness argument it suffices to retain any such patch, since unique
continuation requires only a nonempty relatively open portion of the boundary;
for the quantitative stability estimate the choice of patch of course affects
the constants. Admissible coefficients are taken in
\begin{equation}\label{eq:admissible-set}
\mathcal H_{\mathrm{ad}}
:=\bigl\{h\in C^1(\overline\gamma):\ \inf_\gamma h>0\bigr\},
\end{equation}
an open subset of the Banach space $C^1(\overline\gamma)$; openness is what
makes $h_0+\vartheta$ admissible for all small perturbations $\vartheta$, and
is therefore what the differentiability statements of
Section~\ref{sec:frechet} require. Some statements need only
$h\in L^\infty(\gamma)$ with $h\ge0$, as indicated.

\subsection{Boundary non-degeneracy}\label{sec:boundary_nondegen}

The next proposition shows that the principal eigenfunction is nondegenerate
on the accessible boundary and positive on the inaccessible one, and that its
gradient is bounded below in a collar around the measurement patch.

\begin{proposition}\label{prop:boundary_nondegen}
Let $1<p<\infty$, let $h\in C^1(\overline\gamma)$ with $h\ge0$, and let $u$ be
the normalized positive principal eigenfunction of
\eqref{eq:p-laplace-robin}. Then
\begin{equation}\label{eq:dirichlet-nondegen}
\partial_\nu u(x)<0
\quad\text{and}\quad
|\nabla u(x)|=|\partial_\nu u(x)|>0
\qquad\text{for every }x\in\Gamma_D .
\end{equation}
Moreover $u(x)>0$ for every $x\in\gamma$, and at every $x\in\gamma$ with
$h(x)>0$ one has $\partial_\nu u(x)<0$ and $|\nabla u(x)|>0$. Finally, for
every $\Gamma_{\mathrm m}\Subset\Gamma_D$ there are $c_{\mathrm m}>0$ and an
open neighborhood $\mathcal U_{\mathrm m}$ of $\overline{\Gamma_{\mathrm m}}$
in $\mathbb R^n$ with
\begin{equation}\label{eq:uniform-boundary-nondegeneracy}
|\nabla u(x)|\ge c_{\mathrm m}
\qquad\text{for all }x\in\mathcal U_{\mathrm m}\cap\overline\Omega .
\end{equation}
\end{proposition}

\begin{proof}
By Lemma~\ref{lem:Linfty} we have $u\in L^\infty(\Omega)$, and
$\Gamma_D\cup\gamma\subset\partial\Omega\setminus\zeta$, so
Lemma~\ref{lem:regularity} and Remark~\ref{rem:pointwise-boundary} give the
required $C^1$ boundary regularity and the pointwise boundary conditions
locally on both portions.

Let $x\in\Gamma_D$. Since $u>0$ in $\Omega$ and $u(x)=0$,
Lemma~\ref{lem:hopf} gives $\partial_\nu u(x)<0$. As $u=0$ on the relatively
open set $\Gamma_D$, its tangential derivatives vanish there, so
$\nabla u(x)=\partial_\nu u(x)\nu(x)$, which proves
\eqref{eq:dirichlet-nondegen}.

Now let $x_0\in\gamma$. Continuity gives $u(x_0)\ge0$. If $u(x_0)=0$, then
Lemma~\ref{lem:hopf} yields $\partial_\nu u(x_0)<0$, hence
$F(\nabla u(x_0))\cdot\nu<0$, contradicting the Robin condition
\[
F(\nabla u(x_0))\cdot\nu=-h(x_0)u(x_0)^{p-1}=0 .
\]
Thus $u>0$ on $\gamma$. If $h(x)>0$, then
$F(\nabla u(x))\cdot\nu=-h(x)u(x)^{p-1}<0$, so $\nabla u(x)\neq0$ and
consequently $\partial_\nu u(x)<0$.

Finally, by \eqref{eq:measurement-patch} the compact set
\[
K:=\Bigl\{x\in\overline\Omega:
\operatorname{dist}(x,\overline{\Gamma_{\mathrm m}})\le
\tfrac{d_{\mathrm m}}2\Bigr\}
\]
is disjoint from $\zeta$, so Lemma~\ref{lem:regularity} gives
$\nabla u\in C(K)$, while \eqref{eq:dirichlet-nondegen} gives
$m_{\mathrm m}:=\min_{\overline{\Gamma_{\mathrm m}}}|\nabla u|>0$. By uniform
continuity there is $\delta\in(0,d_{\mathrm m}/2)$ such that
$|\nabla u(x)|\ge m_{\mathrm m}/2$ whenever $x\in\overline\Omega$ and
$\operatorname{dist}(x,\overline{\Gamma_{\mathrm m}})<\delta$. Taking
\[
\mathcal U_{\mathrm m}
:=\bigl\{x\in\mathbb R^n:
\operatorname{dist}(x,\overline{\Gamma_{\mathrm m}})<\delta\bigr\},
\qquad
c_{\mathrm m}:=\frac{m_{\mathrm m}}2,
\]
proves \eqref{eq:uniform-boundary-nondegeneracy}.
\end{proof}

The pointwise conclusions of Proposition~\ref{prop:boundary_nondegen} hold at
every point of $\Gamma_D$ and of $\gamma$. No uniform lower bound up to the
interface $\zeta$ is claimed, and none is needed: the quantitative bound
\eqref{eq:uniform-boundary-nondegeneracy} is required only on the collar where
the Cauchy data are observed.

\subsection{Unconditional consequences of the variational characterization}
\label{sec:eigenvalue-monotonicity}

The characterization \eqref{eq:rayleigh-introduction} expresses $\lambda_1$ as
an infimum of functionals that are affine in $h$, which alone yields a
two-sided comparison. Throughout this subsection $1<p<\infty$.

\begin{proposition}\label{prop:sandwich}
Let $h_1,h_2\in L^\infty(\gamma)$ with $h_j\ge0$, and let $u_j:=u(h_j)$. Then
\begin{equation}\label{eq:sandwich}
\int_\gamma(h_2-h_1)\,u_2^{\,p}\,d\sigma
\;\le\;
\lambda_1(h_2)-\lambda_1(h_1)
\;\le\;
\int_\gamma(h_2-h_1)\,u_1^{\,p}\,d\sigma .
\end{equation}
In particular $h\mapsto\lambda_1(h)$ is nondecreasing and concave on
$\{h\in L^\infty(\gamma):h\ge0\}$.
\end{proposition}

\begin{proof}
Using $u_2$ as a competitor in \eqref{eq:rayleigh-introduction} for the
coefficient $h_1$, and recalling $\|u_2\|_{L^p(\Omega)}=1$,
\[
\lambda_1(h_1)\le\int_\Omega|\nabla u_2|^p\,dx+\int_\gamma h_1u_2^{\,p}\,d\sigma
=\lambda_1(h_2)-\int_\gamma(h_2-h_1)u_2^{\,p}\,d\sigma,
\]
which is the first inequality; exchanging the indices gives the second.
Monotonicity is immediate, and concavity holds because $\lambda_1$ is an
infimum of affine functions of $h$.
\end{proof}

The comparison, together with compactness and simplicity, gives continuity of the principal eigenpair.

\begin{theorem}\label{thm:compactness}
Let $h,h_k\in L^\infty(\gamma)$ with $h,h_k\ge0$ and $h_k\to h$ in
$L^\infty(\gamma)$. Then $\lambda_1(h_k)\to\lambda_1(h)$ and
\[
u(h_k)\to u(h)
\qquad\text{strongly in }W^{1,p}(\Omega),
\]
and consequently $u(h_k)\to u(h)$ strongly in $L^p(\gamma)$.
\end{theorem}

\begin{proof}
Set $u_k:=u(h_k)$, $\lambda_k:=\lambda_1(h_k)$, $u_*:=u(h)$ and
$\lambda:=\lambda_1(h)$, all eigenfunctions being normalized by
$\|\cdot\|_{L^p(\Omega)}=1$ as in Proposition~\ref{prop:principal_eigenvalue}.
Since $h_k\to h$ in $L^\infty(\gamma)$, the sequence $(h_k)$ is uniformly
bounded, and using any fixed $v\in V\setminus\{0\}$ in the Rayleigh quotient
gives $\sup_k\lambda_k<\infty$. Testing \eqref{eq:weak_formulation} with
$\varphi=u_k$ and using the normalization gives the energy identity
\[
\int_\Omega|\nabla u_k|^p\,dx+\int_\gamma h_ku_k^p\,d\sigma=\lambda_k ,
\]
whose left-hand side has nonnegative boundary term because $h_k\ge0$; hence
$\|\nabla u_k\|_{L^p(\Omega)}^p\le\lambda_k$ and $(u_k)$ is bounded in
$W^{1,p}(\Omega)$. By continuity of the trace operator,
\[
C_{\mathrm{tr}}
:=\sup_k\|u_k\|_{L^p(\gamma)}^p+\|u_*\|_{L^p(\gamma)}^p<\infty .
\]
Proposition~\ref{prop:sandwich} then yields
\[
|\lambda_k-\lambda|\le C_{\mathrm{tr}}\|h_k-h\|_{L^\infty(\gamma)}\to0 .
\]

Let $(u_{k_j})$ be an arbitrary subsequence. After passing to a further
subsequence, not relabeled,
\[
u_{k_j}\rightharpoonup\widehat u\ \text{ in }W^{1,p}(\Omega),
\qquad
u_{k_j}\to\widehat u\ \text{ in }L^p(\Omega)\text{ and in }L^p(\gamma),
\]
for some $\widehat u\in V$ with $\widehat u\ge0$ and
$\|\widehat u\|_{L^p(\Omega)}=1$. Since
$\int_\gamma h_{k_j}u_{k_j}^p\,d\sigma\to\int_\gamma h\widehat u^p\,d\sigma$,
weak lower semicontinuity and the energy identity give
\[
\lambda
\le
\int_\Omega|\nabla\widehat u|^p\,dx+\int_\gamma h\widehat u^p\,d\sigma
\le
\liminf_{j\to\infty}\lambda_{k_j}=\lambda,
\]
so $\widehat u$ is a normalized nonnegative principal eigenfunction and
therefore $\widehat u=u_*$ by Proposition~\ref{prop:principal_eigenvalue}.
Furthermore
\[
\int_\Omega|\nabla u_{k_j}|^p\,dx
=\lambda_{k_j}-\int_\gamma h_{k_j}u_{k_j}^p\,d\sigma
\longrightarrow
\int_\Omega|\nabla u_*|^p\,dx,
\]
so uniform convexity of $L^p$ gives $\nabla u_{k_j}\to\nabla u_*$ strongly in
$L^p(\Omega;\mathbb R^n)$, and with the strong $L^p(\Omega)$ convergence,
$u_{k_j}\to u_*$ strongly in $W^{1,p}(\Omega)$.

The limit $u_*$ is independent of the subsequences extracted, so every
subsequence of $(u_k)$ has a further subsequence converging to $u_*$ in
$W^{1,p}(\Omega)$; hence the full sequence converges to $u_*$. The trace
convergence follows by continuity of the trace operator.
\end{proof}

The comparison principle immediately yields uniqueness for ordered
coefficients.

\begin{theorem}\label{thm:ordered-uniqueness}
Let $1<p<\infty$ and let $h_1,h_2\in C^1(\overline\gamma)$ satisfy
$0\le h_1\le h_2$ on $\gamma$. If $\lambda_1(h_1)=\lambda_1(h_2)$, then
$h_1=h_2$ on $\gamma$.
\end{theorem}

\begin{proof}
Let $u_j:=u(h_j)$. By the first inequality in \eqref{eq:sandwich},
\[
\int_\gamma(h_2-h_1)u_2^p\,d\sigma
\le\lambda_1(h_2)-\lambda_1(h_1)=0 .
\]
Since $h_2-h_1\ge0$ and $u_2>0$ on $\gamma$ by
Proposition~\ref{prop:boundary_nondegen}, the integrand is nonnegative, so
$(h_2-h_1)u_2^p=0$ a.e.\ on $\gamma$ and hence $h_1=h_2$ a.e.\ on $\gamma$.
Continuity of $h_1$ and $h_2$ upgrades this to every point of $\gamma$.
\end{proof}

Theorem~\ref{thm:ordered-uniqueness} requires no flux measurement, no unique
continuation argument, and no critical-set nondegeneracy assumption. For
coefficients arising from the thin-layer model,
\eqref{eq:intro-effective-h} gives $h=\rho^{-(p-1)}$, so that $h_1\le h_2$ if
and only if $\rho_1\ge\rho_2$; the ordering corresponds to one coating being
pointwise at least as thick as the other.

The same comparison gives a quantitative weighted $L^1$ stability estimate.

\begin{corollary}\label{cor:L1-stability}
Let $h_1,h_2\in C^1(\overline\gamma)$ satisfy $0\le h_1\le h_2$ on $\gamma$.
Then
\begin{equation}\label{eq:L1-stability}
\bigl\|(h_2-h_1)u_2^p\bigr\|_{L^1(\gamma)}
\le
\lambda_1(h_2)-\lambda_1(h_1)
\le
\bigl\|(h_2-h_1)u_1^p\bigr\|_{L^1(\gamma)} .
\end{equation}
Moreover, for every $K\Subset\gamma$ one has $\min_Ku_2>0$ and
\begin{equation}\label{eq:L1-stability-compact}
\Bigl(\min_K u_2\Bigr)^p\|h_2-h_1\|_{L^1(K)}
\le
\lambda_1(h_2)-\lambda_1(h_1).
\end{equation}
\end{corollary}

\begin{proof}
Since $h_2-h_1\ge0$ and $u_j>0$ on $\gamma$, the two integrals in
\eqref{eq:sandwich} are the $L^1(\gamma)$ norms appearing in
\eqref{eq:L1-stability}. For $K\Subset\gamma$, Proposition~\ref{prop:boundary_nondegen}
and the boundary continuity of $u_2$ give $\min_Ku_2>0$, and restricting the
first integral in \eqref{eq:sandwich} to $K$ gives
\eqref{eq:L1-stability-compact}.
\end{proof}

Estimate~\eqref{eq:L1-stability-compact} is a local Lipschitz bound for
ordered coefficients, obtained without Fr\'echet differentiability and without
any quantitative stability estimate for the linearized inverse problem. The
restriction $K\Subset\gamma$ is natural, because the results used here do not
provide a uniform positive lower bound for $u_2$ up to the interface $\zeta$.

Finally, the same comparison controls the remainder in the first-order
expansion of the eigenvalue map and yields its Fr\'echet derivative, again
without any linearized-invertibility hypothesis. Since the estimate involves
only the sup-norm of the perturbation, it is stated on $L^\infty(\gamma)$.

\begin{proposition}\label{prop:lambda-remainder}
Let $1<p<\infty$, let $h_0\in\mathcal H_{\mathrm{ad}}$, and set
$u_0:=u(h_0)$. Put
\[
\mathcal A:=\bigl\{h\in L^\infty(\gamma):h\ge0\bigr\},
\qquad
\omega_\lambda(s)
:=
\sup\bigl\{
\|u(h)^p-u_0^p\|_{L^1(\gamma)}:
h\in\mathcal A,\
\|h-h_0\|_{L^\infty(\gamma)}\le s
\bigr\}.
\]
Then $\omega_\lambda$ is finite and nondecreasing on bounded intervals,
$\omega_\lambda(s)\to0$ as $s\downarrow0$, and for every $h\in\mathcal A$,
with $t:=\|h-h_0\|_{L^\infty(\gamma)}$,
\begin{equation}\label{eq:lambda-remainder}
\left|
\lambda_1(h)-\lambda_1(h_0)-\int_\gamma(h-h_0)u_0^p\,d\sigma
\right|
\le
t\,\omega_\lambda(t).
\end{equation}
Consequently $\lambda_1$ is Fr\'echet differentiable at $h_0$ as a map from
$L^\infty(\gamma)$ to $\mathbb R$, with
\begin{equation}\label{eq:lambda_derivative}
\lambda_1'(h_0)[\vartheta]
=
\int_\gamma\vartheta\,u_0^p\,d\sigma .
\end{equation}
The same statements hold with $L^\infty(\gamma)$ replaced by
$C^1(\overline\gamma)$, since
$\|\vartheta\|_{L^\infty(\gamma)}\le\|\vartheta\|_{C^1(\overline\gamma)}$ and
$\omega_\lambda$ is nondecreasing.
\end{proposition}

\begin{proof}
Proposition~\ref{prop:sandwich} gives
\[
\int_\gamma(h-h_0)u(h)^p\,d\sigma
\le
\lambda_1(h)-\lambda_1(h_0)
\le
\int_\gamma(h-h_0)u_0^p\,d\sigma ,
\]
whence
\[
\left|
\lambda_1(h)-\lambda_1(h_0)-\int_\gamma(h-h_0)u_0^p\,d\sigma
\right|
\le
\|h-h_0\|_{L^\infty(\gamma)}\,\|u(h)^p-u_0^p\|_{L^1(\gamma)},
\]
which is \eqref{eq:lambda-remainder}.

For the modulus, the elementary inequality
$|a^p-b^p|\le p|a-b|(a^{p-1}+b^{p-1})$ for $a,b\ge0$ and H\"older's
inequality give
\begin{equation}\label{eq:power-difference}
\|u(h)^p-u_0^p\|_{L^1(\gamma)}
\le
p\,\|u(h)-u_0\|_{L^p(\gamma)}
\bigl(\|u(h)\|_{L^p(\gamma)}^{p-1}+\|u_0\|_{L^p(\gamma)}^{p-1}\bigr).
\end{equation}
As in the proof of Theorem~\ref{thm:compactness}, the traces $u(h)$ are
bounded in $L^p(\gamma)$ uniformly for $h$ in a bounded subset of
$\mathcal A$, so $\omega_\lambda$ is finite; monotonicity is clear from the
definition. Suppose $\omega_\lambda(s)\not\to0$. Then there are
$\eta>0$, $s_k\downarrow0$ and $h_k\in\mathcal A$ with
$\|h_k-h_0\|_{L^\infty(\gamma)}\le s_k$ and
$\|u(h_k)^p-u_0^p\|_{L^1(\gamma)}\ge\eta$. Since $h_k\to h_0$ in
$L^\infty(\gamma)$, Theorem~\ref{thm:compactness} gives $u(h_k)\to u_0$ in
$L^p(\gamma)$, and \eqref{eq:power-difference} forces
$\|u(h_k)^p-u_0^p\|_{L^1(\gamma)}\to0$, a contradiction. Hence
$\omega_\lambda(s)\to0$.

 Finally, $\vartheta\mapsto\int_\gamma\vartheta\,u_0^p\,d\sigma$ is a bounded linear functional on $L^\infty(\gamma)$, because $u_0^p\in L^1(\gamma)$. Since $h_0\in\mathcal H_{\mathrm{ad}}$, one has $\inf_\gamma h_0>0$, so the ball $\{h\in L^\infty(\gamma):\|h-h_0\|_{L^\infty(\gamma)}<\inf_\gamma h_0\}$ is contained in $\mathcal A$, and Fr\'echet differentiability at $h_0$ is understood on this $L^\infty$-neighborhood, on which $\lambda_1$ is defined. With this reading,  \eqref{eq:lambda-remainder} is exactly the statement that $\lambda_1$ is Fr\'echet differentiable at $h_0$ with derivative  \eqref{eq:lambda_derivative}.
\end{proof}

The next lemma specializes the comparison to a ray of coefficients. It is used
in Section~\ref{sec:verification}, where the structural hypotheses are
examined on explicit geometries and the recovery of a single scalar parameter
is at issue.

\begin{lemma}\label{lem:one-parameter-family}
Let $1<p<\infty$, let $\bar h\in C^1(\overline\gamma)$, and let
$0\not\equiv\psi\in C^1(\overline\gamma)$ satisfy $\psi\ge0$. Assume
\[
h_\theta:=\bar h+\theta\psi\in\mathcal H_{\mathrm{ad}}
\qquad(\theta\in I),
\]
where $I\subset\mathbb R$ is an open interval, and set
$\lambda(\theta):=\lambda_1(h_\theta)$ and $u_\theta:=u(h_\theta)$. Then, for
$\theta_1<\theta_2$,
\begin{equation}\label{eq:one-parameter-comparison}
(\theta_2-\theta_1)\int_\gamma\psi u_{\theta_2}^p\,d\sigma
\le
\lambda(\theta_2)-\lambda(\theta_1)
\le
(\theta_2-\theta_1)\int_\gamma\psi u_{\theta_1}^p\,d\sigma .
\end{equation}
In particular $\lambda$ is strictly increasing and
\begin{equation}\label{eq:one-parameter-derivative}
\lambda'(\theta)=\int_\gamma\psi u_\theta^p\,d\sigma>0 ,
\end{equation}
so $\lambda(\theta)$ determines $\theta$ uniquely. Moreover, for every compact
interval $J\Subset I$ there is $C_J>0$ with
\begin{equation}\label{eq:one-parameter-lipschitz}
|\theta_2-\theta_1|\le C_J|\lambda(\theta_2)-\lambda(\theta_1)|,
\qquad\theta_1,\theta_2\in J .
\end{equation}
\end{lemma}

\begin{proof}
For $\theta_1<\theta_2$ one has
$h_{\theta_2}-h_{\theta_1}=(\theta_2-\theta_1)\psi\ge0$, so
Proposition~\ref{prop:sandwich} gives
\eqref{eq:one-parameter-comparison}. Since $\psi\ge0$ is nontrivial and
$u_\theta>0$ on $\gamma$ by Proposition~\ref{prop:boundary_nondegen}, we have
$\int_\gamma\psi u_\theta^p\,d\sigma>0$, so $\lambda$ is strictly increasing;
and by Proposition~\ref{prop:lambda-remainder},
$\lambda'(\theta)=\lambda_1'(h_\theta)[\psi]
=\int_\gamma\psi u_\theta^p\,d\sigma$, which is
\eqref{eq:one-parameter-derivative}. Finally
$m(\theta):=\int_\gamma\psi u_\theta^p\,d\sigma$ is positive and, by
Theorem~\ref{thm:compactness}, continuous in $\theta$; hence
$m_J:=\min_{\theta\in J}m(\theta)>0$ and
\eqref{eq:one-parameter-comparison} gives
$|\lambda(\theta_2)-\lambda(\theta_1)|\ge m_J|\theta_2-\theta_1|$, that is
\eqref{eq:one-parameter-lipschitz} with $C_J=m_J^{-1}$.
\end{proof}

\subsection{The linearized coefficients and the critical sets}
\label{sec:linearized}

From here to the end of Section~\ref{Inv} 
we assume $p\ge2$. Let $h_1,h_2\in\mathcal H_{\mathrm{ad}}$ and $u_j:=u(h_j)$. With $F$ and $G$
as in Section~\ref{sec:preliminaries},
\[
F:\mathbb R^n\to\mathbb R^n, \qquad 
F(z)=|z|^{p-2}z, \qquad   
\]
\[
G:\mathbb R\to\mathbb R, \qquad G(s):=|s|^{p-2}s.
\]
Define the \emph{path coefficients}
\begin{align}
\overline A(x)
&:=\int_0^1DF\bigl((1-t)\nabla u_2(x)+t\nabla u_1(x)\bigr)\,dt,
\label{eq:path-coefficient}\\
\overline B(x)
&:=\int_0^1G'\bigl((1-t)u_2(x)+tu_1(x)\bigr)\,dt
=(p-1)\int_0^1\bigl((1-t)u_2(x)+tu_1(x)\bigr)^{p-2}\,dt.
\label{eq:path-scalar-coefficient}
\end{align}
The second expression uses $G'(s)=(p-1)|s|^{p-2}$ and the positivity of $u_1,u_2$ in $\Omega$, by which $(1-t)u_2+tu_1>0$ for every $t\in[0,1]$, so the absolute value may be omitted. Since $DF(z)$ is
symmetric, so is $\overline A(x)$, and by Lemma~\ref{lem:Linfty},
\begin{equation}\label{eq:Bbar-bounded}
0\le\overline B(x)\le(p-1)\bigl(u_1(x)+u_2(x)\bigr)^{p-2},
\qquad\overline B\in L^\infty(\Omega).
\end{equation}
Because $F$ and $G$ are $C^1$ for $p\ge2$, the fundamental theorem of calculus
gives, with $w:=u_1-u_2$,
\begin{equation}\label{eq:path-identities}
F(\nabla u_1)-F(\nabla u_2)=\overline A\,\nabla w,
\qquad
G(u_1)-G(u_2)=\overline B\,w .
\end{equation}
If moreover $\lambda_1(h_1)=\lambda_1(h_2)=:\lambda$, subtracting the two weak
eigenvalue equations gives the \emph{exact difference equation}
\begin{equation}\label{eq:exact-difference-equation}
-\operatorname{div}\bigl(\overline A(x)\nabla w\bigr)=\lambda\,\overline B(x)\,w
\qquad\text{in }\Omega,
\end{equation}
in the distributional sense in $\Omega$. Since $p\ge2$ and $\Omega$ is
bounded, $W^{1,p}(\Omega)\hookrightarrow H^1(\Omega)$, so $w\in H^1(\Omega)$.

By Lemma~\ref{lem:path-matrix}, with
\begin{equation}\label{eq:M-definition}
M(x):=\max\bigl\{|\nabla u_1(x)|,|\nabla u_2(x)|\bigr\},
\end{equation}
one has $\overline A\equiv I$ when $p=2$, while for $p>2$ the matrix
$\overline A$ is comparable to $M^{p-2}I$ and hence degenerates exactly where
both gradients vanish. Accordingly we set
\begin{equation}\label{eq:critical-sets}
\Sigma_j:=\{x\in\Omega:\nabla u_j(x)=0\},
\qquad
\Sigma_\cup:=\Sigma_1\cup\Sigma_2,
\qquad
\Sigma_\cap:=\Sigma_1\cap\Sigma_2 .
\end{equation}

The next lemma records that the critical sets are closed and cannot contain a
ball, so that their complements are dense.

\begin{lemma}\label{lem:critical-sets}
Each $\Sigma_j$ is relatively closed in $\Omega$ with empty interior. The same
holds for $\Sigma_\cup$ and $\Sigma_\cap$, whose complements in $\Omega$ are
therefore open and dense.
\end{lemma}

\begin{proof}
Closedness follows from continuity of $\nabla u_j$ in $\Omega$. If
$\nabla u_j$ vanished on a ball $B\subset\Omega$, then $u_j\equiv c$ on $B$
and the equation would give $0=\lambda_1(h_j)c^{p-1}$ in $\mathcal D'(B)$,
hence $c=0$, contradicting $u_j>0$. For the union, if a ball
$B\subset\Omega$ satisfied $B\subset\Sigma_1\cup\Sigma_2$, then
$B\setminus\Sigma_1$ would be an open subset of $\Sigma_2$; it is nonempty
because $\Sigma_1$ has empty interior, contradicting the fact that
$\Sigma_2$ has empty interior as well. Finally
$\Sigma_\cap\subset\Sigma_1$.
\end{proof}
 
The next lemma supplies the two properties of $\overline A$ used in the propagation argument: away from the common critical set it is uniformly
elliptic, and away from the union of the critical sets it is in addition Lipschitz.
\begin{lemma}\label{lem:local-coefficients}
Let $p\ge2$.
\begin{enumerate}[(i)]
\item
If $K\subset\Omega\setminus\Sigma_\cap$ is compact, then $\overline A$ is
uniformly elliptic on $K$, with constants depending only on $p$, $\min_KM$
and $\max_KM$.

\item
If $K\subset\Omega\setminus\Sigma_\cup$ is compact, then $u_1,u_2$ are smooth
in a neighborhood of $K$ and $\overline A|_K$ is Lipschitz.
\end{enumerate}
\end{lemma}

\begin{proof}
Part~(i) follows from the continuity and positivity of $M$ on $K$ together
with Lemma~\ref{lem:path-matrix}.

For~(ii), continuity of the gradients gives a neighborhood of $K$ on which
$|\nabla u_j|\ge m/2$, where
\[
m:=\min_K\min_j|\nabla u_j|>0,
\]
hence
Lemma~\ref{lem:nondegenerate-regularity}(i) gives smoothness of $u_1,u_2$
near $K$. Set
\[
L_K:=\max_j\operatorname{Lip}_K(\nabla u_j)<\infty,
\qquad
M_K:=\max_KM,
\qquad
v_x(t):=(1-t)\nabla u_2(x)+t\nabla u_1(x),
\]
so that $|v_x(t)-v_y(t)|\le2L_K|x-y|$ for $x,y\in K$ and $t\in[0,1]$.

For $p=2$ we have $\overline A\equiv I$. Let $p>2$. By
\eqref{eq:DF-increment-max},
\[
\|\overline A(x)-\overline A(y)\|
\le
2C_pL_K|x-y|
\int_0^1\max\{|v_x(t)|,|v_y(t)|\}^{p-3}\,dt .
\]
If $p\ge3$, the integral is bounded by $M_K^{p-3}$. If $2<p<3$, then $p-3<0$,
so
\[
\max\{|v_x(t)|,|v_y(t)|\}^{p-3}\le|v_x(t)|^{p-3},
\]
and it suffices to
bound $\int_0^1|v_x(t)|^{p-3}\,dt$ uniformly on $K$. For each $x\in K$ either
$\min_{t\in[0,1]}|v_x(t)|\ge m/2$, in which case the integral is at most
$(m/2)^{p-3}$, or $|\nabla u_1(x)-\nabla u_2(x)|\ge m/2$, since otherwise the
segment would remain within distance $m/2$ of $\nabla u_2(x)$ and hence stay
outside the ball of radius $m/2$. In the latter case, set 
\[
\kappa_x:=|\nabla u_1(x)-\nabla u_2(x)|\in[m/2,2M_K].
\]
Then $|v_x(t)|^2=d_x^2+\kappa_x^2|t-t_x|^2$ for suitable $d_x\ge0$ and
$t_x\in\mathbb R$, so the substitution $s=\kappa_x(t-t_x)$ gives
$|v_x(t)|\ge|s|$ and $dt=ds/\kappa_x\le(2/m)\,ds$, whence
\[
\int_0^1|v_x(t)|^{p-3}\,dt
\le\frac2m\int_{-2M_K}^{2M_K}|s|^{p-3}\,ds
=\frac{2}{m}\cdot\frac{2\,(2M_K)^{p-2}}{p-2},
\]
which is finite because $p-3>-1$. In either case
\[
\sup_{x\in K}\int_0^1|v_x(t)|^{p-3}\,dt\le C_{p,m,M_K}<\infty,
\]
 and the
required Lipschitz bound follows.
\end{proof}

\subsection{Uniqueness from the eigenvalue and the boundary flux}
\label{sec:uniqueness}

The following assumption is the hypothesis under which the general uniqueness
theorem is proved; it concerns the topology of the complement of the critical
sets and is stated separately for the plane and for higher dimensions.

\begin{assumption}\label{ass:connectivity}
Let $p\ge2$, let $h_1,h_2\in\mathcal H_{\mathrm{ad}}$, and let $u_1,u_2$ be
the corresponding normalized positive principal eigenfunctions. With the
notation \eqref{eq:critical-sets}, assume that the open set
\[
D:=
\begin{cases}
\Omega, & p=2,\\[1mm]
\Omega\setminus\Sigma_\cap, & p>2,\ n=2,\\[1mm]
\Omega\setminus\Sigma_\cup, & p>2,\ n\ge3,
\end{cases}
\]
is connected.
\end{assumption}

For $p=2$ one has $D=\Omega$, so the assumption is vacuous: the critical sets
may well be nonempty, but they are irrelevant because $\overline A\equiv I$.
For $p>2$ the assumption holds whenever the relevant critical set is
$\mathcal H^{n-1}$-negligible, by the following lemma.

\begin{lemma}\label{lem:separator}
Let $\Omega\subset\mathbb R^n$ be open and connected and let
$\Sigma\subset\Omega$ be relatively closed with
$\mathcal H^{n-1}(\Sigma)=0$. Then $\Omega\setminus\Sigma$ is connected.
\end{lemma}

\begin{proof}
First let $B\Subset\Omega$ be a ball and let $x,y\in B\setminus\Sigma$. For
$a\in B\setminus\Sigma$ define
$E_a:=\{z\in B:[a,z]\cap\Sigma\neq\emptyset\}$. The radial projection
$\pi_a(\xi):=(\xi-a)/|\xi-a|$ is Lipschitz on each set
$\{|\xi-a|\ge1/k\}$, so $\mathcal H^{n-1}(\pi_a(\Sigma))=0$, and polar
coordinates give $\mathcal L^n(E_a)=0$. Choosing
$z\in B\setminus(E_x\cup E_y)$ gives
$[x,z]\cup[z,y]\subset B\setminus\Sigma$, so $B\setminus\Sigma$ is path
connected.

For arbitrary $x,y\in\Omega\setminus\Sigma$, join them by a continuous path in
$\Omega$ and cover its compact image by a finite chain of balls
$B_i\Subset\Omega$ with consecutive overlaps. Since
$\mathcal H^{n-1}(\Sigma)=0$, the set $\Sigma$ has empty interior, so each
overlap contains a point outside $\Sigma$. Applying the first part in each
ball and concatenating the resulting paths shows that $\Omega\setminus\Sigma$
is path connected, hence connected.
\end{proof}

We can now state the main uniqueness result: the principal eigenvalue together
with the nonlinear flux on the measurement patch determines the Robin
coefficient.

\begin{theorem}\label{thm:uniqueness}
Let $p\ge2$, and let $\Omega$, $\Gamma_D$, and $\gamma$ satisfy the standing
assumptions of Section~\ref{sec:problem-setting}. If $p>2$ and $n\ge3$,
assume additionally that $\partial\Omega$ is of class $C^{2,\alpha_0}$ in a
neighborhood of $\overline{\Gamma_{\mathrm m}}$ for some $\alpha_0\in(0,1)$.
Let $h_1,h_2\in\mathcal H_{\mathrm{ad}}$ have normalized positive principal
eigenpairs $(\lambda_1(h_j),u_j)$, and suppose
Assumption~\ref{ass:connectivity} holds for $(u_1,u_2)$. If
\begin{equation}\label{eq:uniqueness_data}
\lambda_1(h_1)=\lambda_1(h_2)
\qquad\text{and}\qquad
q_{h_1}=q_{h_2}\quad\text{on }\Gamma_{\mathrm m}
\end{equation}
in the conormal-trace sense, then $h_1=h_2$ on $\gamma$.
\end{theorem}

\begin{proof}
Set $\lambda:=\lambda_1(h_1)=\lambda_1(h_2)$ and $w:=u_1-u_2$, so that
\eqref{eq:exact-difference-equation} holds and $w\in H^1(\Omega)$. Both
eigenfunctions vanish on $\Gamma_D$, so $w=0$ on $\Gamma_{\mathrm m}$, and by
\eqref{eq:path-identities} the flux equality in \eqref{eq:uniqueness_data}
gives
\begin{equation}\label{eq:conormal-vanishes}
(\overline A\nabla w)\cdot\nu
=\bigl(F(\nabla u_1)-F(\nabla u_2)\bigr)\cdot\nu
=q_{h_1}-q_{h_2}=0
\qquad\text{on }\Gamma_{\mathrm m},
\end{equation}
and by Lemma~\ref{lem:regularity} and Remark~\ref{rem:pointwise-boundary} this
pointwise identity coincides with the weak conormal trace.

By Proposition~\ref{prop:boundary_nondegen}, applied to $u_1$ and to $u_2$,
after intersecting the two collars and decreasing the constant we may choose
$\mathcal U_{\mathrm m}$ and $c_{\mathrm m}>0$ with
\[
|\nabla u_j|\ge c_{\mathrm m}
\qquad\text{on }\mathcal U_{\mathrm m}\cap\overline\Omega,
\quad j=1,2 .
\]
Choose a connected open set $U\subset\mathbb R^n$ with
\[
\emptyset\neq U\cap\partial\Omega\subset\Gamma_{\mathrm m},
\qquad
\overline{U\cap\Omega}\subset\mathcal U_{\mathrm m}\cap\overline\Omega .
\]
Lemma~\ref{lem:path-matrix} gives uniform ellipticity of $\overline A$ on
$U\cap\Omega$. If $n\ge3$, then $\overline A$ is in addition Lipschitz on
$\overline{U\cap\Omega}$: this is immediate for $p=2$, while for $p>2$ it
follows from the local $C^{2,\alpha}$ regularity up to
$\Gamma_{\mathrm m}$ furnished by
Lemma~\ref{lem:nondegenerate-regularity}(ii), which is where the
$C^{2,\alpha_0}$ hypothesis on $\partial\Omega$ is used, combined with the
averaging estimate in the proof of
Lemma~\ref{lem:local-coefficients}(ii), which applies unchanged on this closed
boundary collar.

Moreover $\overline A\nabla w=F(\nabla u_1)-F(\nabla u_2)$ belongs to
$L^2(U\cap\Omega;\mathbb R^n)$, its divergence belongs to $L^2(U\cap\Omega)$
by \eqref{eq:exact-difference-equation}, and its weak conormal trace vanishes
on $U\cap\partial\Omega$ by \eqref{eq:conormal-vanishes}. Hence the local
Gauss--Green formula gives
\[
\int_{U\cap\Omega}\overline A\nabla w\cdot\nabla\varphi\,dx
=\lambda\int_{U\cap\Omega}\overline B\,w\,\varphi\,dx,
\qquad\varphi\in C_c^\infty(U),
\]
which is \eqref{eq:ucp-local-weak}. Theorem~\ref{thm:boundary_ucp_local}, in
dimension two in the form recorded in Remark~\ref{rem:ucp-2d}, therefore gives
\[
w\equiv0\qquad\text{in }U\cap\Omega .
\]
The set $U\cap\Omega$ is nonempty, open, and contained in $D$, since
$|\nabla u_j|\ge c_{\mathrm m}>0$ there.

By Assumption~\ref{ass:connectivity} the set $D$ is connected; by
Lemma~\ref{lem:local-coefficients}, $\overline A$ is locally uniformly
elliptic in $D$, and locally Lipschitz there when $p=2$ or $n\ge3$; and
$\overline B\in L^\infty(\Omega)$ by \eqref{eq:Bbar-bounded}. Applying
Theorem~\ref{thm:interior_propagation}, in dimension two in the form recorded
in Remark~\ref{rem:ucp-2d}, to \eqref{eq:exact-difference-equation} gives
\[
w\equiv0\qquad\text{in }D .
\]

By Lemma~\ref{lem:critical-sets} the set $D$ is dense in $\Omega$, and $w$ is
continuous, so $u_1=u_2=:u$ in $\Omega$. Let $x_0\in\gamma$. Since
$x_0\notin\zeta$, Lemma~\ref{lem:regularity} gives $C^1$ regularity up to the
boundary near $x_0$, so $\nabla u_1(x_0)=\nabla u_2(x_0)$, and by
Remark~\ref{rem:pointwise-boundary} both Robin conditions hold pointwise at
$x_0$. Subtracting them gives
$(h_1(x_0)-h_2(x_0))u(x_0)^{p-1}=0$, and $u(x_0)>0$ by
Proposition~\ref{prop:boundary_nondegen}, so $h_1(x_0)=h_2(x_0)$. As
$x_0\in\gamma$ was arbitrary, $h_1=h_2$ on $\gamma$.
\end{proof}

For $p=2$, Assumption~\ref{ass:connectivity} is vacuous and
Theorem~\ref{thm:uniqueness} recovers the linear inverse Robin spectral
setting and the uniqueness mechanism studied by Santacesaria and
Yachimura~\cite{sy2020}.

The proof consumed exactly two properties of $\overline A$: uniform ellipticity
on the boundary collar and in $D$, and enough regularity for unique
continuation. The next remark explains why the first cannot be bought
artificially when the second is unavailable.

\begin{remark}\label{rem:localized-regularization}
A cutoff regularization does not remove the obstruction. Let $U$ be the
boundary collar of the preceding proof, choose $U^0\Subset U^1\Subset U$ with
$\overline{\Gamma_{\mathrm m}}\cap\Omega\subset U^0$, and let
$\chi\in C^\infty(\overline\Omega)$ satisfy $\chi=0$ on $U^0$ and $\chi=1$
outside $U^1$. Then
\[
\overline A_\tau:=\overline A+\tau\chi I
\]
agrees with $\overline A$ near $\Gamma_{\mathrm m}$ and is uniformly elliptic
on $\Omega$: on $U^1\cap\Omega$ this follows from
Proposition~\ref{prop:boundary_nondegen} and
Lemma~\ref{lem:path-matrix}, while outside $U^1$ one has
$\overline A_\tau\ge\tau I$. However,
\[
-\operatorname{div}(\overline A_\tau\nabla w)
=\lambda\overline B\,w-\tau\operatorname{div}(\chi\nabla w),
\]
so $w$ no longer satisfies a homogeneous equation of the required form.
Moreover, adding the smooth term $\tau\chi I$ does not improve the regularity
of $\overline A$ needed for interior propagation in dimensions $n\ge3$, which
by Remark~\ref{rem:ucp-regularity} cannot simply be dispensed with.
\end{remark}

\subsection{Differentiability of the principal branch}\label{sec:frechet}

Fix $h_0\in\mathcal H_{\mathrm{ad}}$ with normalized positive principal
eigenpair $(\lambda_0,u_0)$, and define the normalization tangent hyperplane
\begin{equation}\label{eq:Vperp}
V_\perp
:=
\Bigl\{v\in V:\int_\Omega u_0^{p-1}v\,dx=0\Bigr\}.
\end{equation}
Perturbation directions are denoted $\vartheta\in C^1(\overline\gamma)$. The
eigenvalue component of the expansion is settled by
Proposition~\ref{prop:lambda-remainder}; only the eigenfunction requires
further input.

Define the residual map
\begin{equation}\label{eq:residual-map}
\langle\mathcal R(h,u,\lambda),\phi\rangle
:=\int_\Omega|\nabla u|^{p-2}\nabla u\cdot\nabla\phi\,dx
-\lambda\int_\Omega|u|^{p-2}u\,\phi\,dx
+\int_\gamma h|u|^{p-2}u\,\phi\,d\sigma,
\qquad\phi\in V,
\end{equation}
and the normalization constraint
\begin{equation}\label{eq:constraint-map}
\mathcal C(u):=\|u\|_{L^p(\Omega)}^p-1 ;
\end{equation}
the augmented map
\begin{equation}\label{eq:augmented-map}
\mathbb H(h,u,\lambda):=\bigl(\mathcal R(h,u,\lambda),\mathcal C(u)\bigr)
:\ \mathcal H_{\mathrm{ad}}\times V\times\mathbb R\to V^*\times\mathbb R
\end{equation}
vanishes precisely at normalized eigenpairs of \eqref{eq:p-laplace-robin}.
That $\mathbb H$ is continuously differentiable is not a formal consequence of
the pointwise smoothness of $F$, since differentiability of a superposition
operator between function spaces is a statement about those spaces.

\begin{lemma}\label{lem:superposition}
Let $p\ge2$, let $p':=p/(p-1)$, and let $\Omega$ be bounded. The
superposition operator induced by $F$,
\[
\mathcal N:L^p(\Omega;\mathbb R^n)\to L^{p'}(\Omega;\mathbb R^n),
\qquad
\mathcal N(z):=F\circ z=|z|^{p-2}z,
\]
is of class $C^1$, and its derivative is multiplication by the matrix field
$DF(z)$, that is $D\mathcal N(z)w=DF(z)w$ pointwise a.e. The same holds for the scalar map on $L^p(\Omega)$
and on $L^p(\gamma)$. Consequently, $\mathcal R$, $\mathcal C$ and
$\mathbb H$ are continuously Fr\'echet differentiable.
\end{lemma}

\begin{proof}
For $p=2$ the map is the identity, so let $p>2$. By \eqref{eq:DF-two-sided},
$\|DF(z)\|\le(p-1)|z|^{p-2}$, and H\"older with the exponents $p/(p-2)$ and
$p$, whose reciprocals sum to $1/p'$, gives
\begin{equation}\label{eq:superposition-derivative-bound}
\|DF(z)w\|_{L^{p'}}\le(p-1)\|z\|_{L^p}^{p-2}\|w\|_{L^p}.
\end{equation}
Write the remainder in integral form,
\[
F(z+w)-F(z)-DF(z)w=\int_0^1\bigl(DF(z+sw)-DF(z)\bigr)w\,ds,
\]
and estimate by \eqref{eq:DF-increment-cases}. For $2<p\le3$ the integrand is
bounded by $C_p|w|^{p-1}$, and since $(p-1)p'=p$ its $L^{p'}$ norm is
$C_p\|w\|_{L^p}^{p-1}=o(\|w\|_{L^p})$ because $p-1>1$. For $p\ge3$ the
integrand is bounded by $C_p(|z|+|w|)^{p-3}|w|^2$, and H\"older with the
exponents $p/(p-3)$ and $p/2$ gives
\[
C_p(\|z\|_{L^p}+\|w\|_{L^p})^{p-3}\|w\|_{L^p}^2=o(\|w\|_{L^p}).
\]

The same increment estimates and H\"older's inequality give
\[
\|D\mathcal N(z_1)-D\mathcal N(z_2)\|_{\mathcal L(L^p,L^{p'})}
\le
\begin{cases}
C_p\|z_1-z_2\|_{L^p}^{p-2}, & 2<p<3,\\[1mm]
C_p\bigl(\|z_1\|_{L^p}+\|z_2\|_{L^p}\bigr)^{p-3}\|z_1-z_2\|_{L^p},
& p\ge3,
\end{cases}
\]
so $z\mapsto D\mathcal N(z)$ is continuous. The scalar and boundary statements
follow identically, using the bounded trace $V\to L^p(\gamma)$ and the
multiplication $L^\infty(\gamma)\times L^{p'}(\gamma)\to L^{p'}(\gamma)$.
Consequently $\mathcal R$ is $C^1$ by bounded duality pairing, and
$\mathcal C$ is $C^1$ with
$D\mathcal C(u)[v]=p\int_\Omega|u|^{p-2}u\,v\,dx$.
\end{proof}

Differentiating at the normalized principal eigenpair
$(h,u,\lambda)=(h,u(h),\lambda_1(h))$ gives
\begin{align}
\langle D_h\mathcal R[\vartheta],\phi\rangle
&=\int_\gamma\vartheta|u|^{p-2}u\,\phi\,d\sigma,
\label{eq:DhR}\\
\langle D_u\mathcal R[v],\phi\rangle
&=\int_\Omega A\nabla v\cdot\nabla\phi\,dx
-\lambda(p-1)\int_\Omega|u|^{p-2}v\phi\,dx
+(p-1)\int_\gamma h|u|^{p-2}v\phi\,d\sigma
=:\langle\mathcal M_h[v],\phi\rangle,
\label{eq:DuR}\\
\langle D_\lambda\mathcal R[\mu],\phi\rangle
&=-\mu\int_\Omega|u|^{p-2}u\,\phi\,dx,
\qquad
D\mathcal C(u)[v]=p\int_\Omega|u|^{p-2}u\,v\,dx,
\label{eq:DlamR}
\end{align}
where $A(x):=DF(\nabla u(x))$; we write $\mathcal M_0:=\mathcal M_{h_0}$ and
$A_0:=DF(\nabla u_0)$. Since $A(x)$ is symmetric, the bilinear form in
\eqref{eq:DuR} is symmetric, that is
$\langle\mathcal M_h[v],\phi\rangle=\langle\mathcal M_h[\phi],v\rangle$ for
$v,\phi\in V$; this is used twice below. We also identify $u_0^{\,p-1}$ with
the functional $\phi\mapsto\int_\Omega u_0^{\,p-1}\phi\,dx$ in $V^*$.

The next lemma records that the homogeneity of the equation puts the
eigenfunction itself in the kernel of the linearized operator, which is why
the normalization has to be adjoined.

\begin{lemma}\label{lem:kernel-direction}
For every $h\in\mathcal H_{\mathrm{ad}}$ one has $\mathcal M_h[u(h)]=0$ in
$V^*$.
\end{lemma}

\begin{proof}
Since $DF(z)z=(p-1)F(z)$, we get $A\nabla u=(p-1)|\nabla u|^{p-2}\nabla u$;
substituting $v=u$ in \eqref{eq:DuR} gives $(p-1)$ times
\eqref{eq:weak_formulation}, hence zero.
\end{proof}

Lemma~\ref{lem:kernel-direction} shows that the eigenvalue equation alone can
never have an invertible linearization. One is thus led to
\begin{equation}\label{eq:augmented-linearization}
\mathcal L_0(v,\mu)
:=\Bigl(\mathcal M_0[v]-\mu u_0^{\,p-1},\
p\int_\Omega u_0^{\,p-1}v\,dx\Bigr)
:\ V\times\mathbb R\to V^*\times\mathbb R,
\end{equation}
which is $D_{(u,\lambda)}\mathbb H(h_0,u_0,\lambda_0)$.

The following remark shows that, for exponents above two, no operator between
the natural pair of spaces can be an isomorphism, so that the standard
implicit function route is unavailable for reasons of Banach space geometry.

\begin{remark}\label{rem:no-isomorphism}
Let $p>2$ and $V=W^{1,p}_{\Gamma_D}(\Omega)$. Then $V\times\mathbb R$ and
$V^*\times\mathbb R$ are not isomorphic, so
\eqref{eq:augmented-linearization} cannot be an isomorphism.

Indeed, the map $v\mapsto(v,\nabla v)$ identifies $V$ with a closed subspace
of $L^p(\Omega;\mathbb R^{n+1})$. Since $L^p$ has Rademacher type $2$ for
$p\ge2$ and type passes to subspaces, $V$, and hence $V\times\mathbb R$, has
type $2$. By the duality principle for type and cotype, the dual of a space of
type $2$ has cotype $2$, so $V^*\times\mathbb R$ has cotype $2$; see
\cite[Ch.~6]{AlbiacKalton16} and \cite[Ch.~11]{DiestelJarchowTonge95}. If
these spaces were isomorphic, one space would have both type $2$ and cotype
$2$ and would therefore be isomorphic to a Hilbert space by Kwapie\'n's
theorem \cite{Kwapien72}.

On the other hand, choosing normalized functions
$\varphi_k\in C_c^\infty(\Omega)$ with pairwise disjoint supports shows that
their closed span in $V$ is isomorphic to $\ell^p$. Since $\ell^p$ is not
Hilbertian for $p\neq2$, neither $V\times\mathbb R$ nor any space isomorphic
to it can be Hilbertian, a contradiction.
\end{remark}

We therefore do not assume invertibility of $\mathcal L_0$, and instead assume
directly the property that the stability argument uses.

\begin{assumption}\label{ass:frechet-branch}
Let $p\ge2$ and $h_0\in\mathcal H_{\mathrm{ad}}$. There exist a neighborhood
$\mathcal U\subset\mathcal H_{\mathrm{ad}}$ of $h_0$ in
$C^1(\overline\gamma)$, a bounded linear map
$u'(h_0):C^1(\overline\gamma)\to V$, and a nondecreasing function
$\omega_u:[0,\infty)\to[0,\infty)$ with $\omega_u(s)\to0$ as $s\downarrow0$,
such that for every $h\in\mathcal U$,
\begin{equation}\label{eq:frechet-remainder-u}
\bigl\|u(h)-u_0-u'(h_0)[h-h_0]\bigr\|_{W^{1,p}(\Omega)}
\ \le\
\omega_u\bigl(\|h-h_0\|_{C^1(\overline\gamma)}\bigr)\,
\|h-h_0\|_{C^1(\overline\gamma)} .
\end{equation}
Equivalently, the map $h\mapsto u(h)$ is Fr\'echet differentiable at $h_0$
from $C^1(\overline\gamma)$ into $W^{1,p}(\Omega)$, with derivative
$u'(h_0)$; the modulus $\omega_u$ is then recovered as in
\eqref{eq:IFT-modulus}.
\end{assumption}

Combining \eqref{eq:frechet-remainder-u} with
Proposition~\ref{prop:lambda-remainder}, and using
$\|h-h_0\|_{L^\infty(\gamma)}\le\|h-h_0\|_{C^1(\overline\gamma)}$ together
with the monotonicity of $\omega_\lambda$, gives for every
$h\in\mathcal U$
\begin{equation}\label{eq:frechet-remainder}
\bigl\|u(h)-u_0-u'(h_0)[h-h_0]\bigr\|_{W^{1,p}(\Omega)}
+\bigl|\lambda_1(h)-\lambda_0-\lambda_1'(h_0)[h-h_0]\bigr|
\ \le\ \omega\bigl(\|h-h_0\|_{C^1(\overline\gamma)}\bigr)\,
\|h-h_0\|_{C^1(\overline\gamma)},
\end{equation}
with $\omega:=\omega_u+\omega_\lambda$.

In the linear case the assumption is a theorem.

\begin{proposition}\label{prop:frechet-p2}
For $p=2$, Assumption~\ref{ass:frechet-branch} holds at every
$h_0\in\mathcal H_{\mathrm{ad}}$.
\end{proposition}

\begin{proof}
For $p=2$ we have $V=H^1_{\Gamma_D}(\Omega)$ and $A_0\equiv I$. The form
\[
a_0(v,\phi):=\int_\Omega\nabla v\cdot\nabla\phi\,dx
+\int_\gamma h_0v\phi\,d\sigma
\]
is bounded and coercive on $V$, while $V\hookrightarrow L^2(\Omega)$ is
compact. Hence
$\mathcal M_0=a_0-\lambda_0(\,\cdot\,,\,\cdot\,)_{L^2(\Omega)}$ is Fredholm of
index zero, and by \eqref{eq:full-simplicity},
$\ker\mathcal M_0=\operatorname{span}\{u_0\}$, so that symmetry gives
$\operatorname{ran}\mathcal M_0=\{f\in V^*:\langle f,u_0\rangle=0\}$.

For injectivity of $\mathcal L_0$, let $\mathcal L_0(v,\mu)=(0,0)$. Pairing
$\mathcal M_0[v]=\mu u_0$ with $u_0$, and using the symmetry of
$\mathcal M_0$ together with Lemma~\ref{lem:kernel-direction}, gives
$\mu\|u_0\|_{L^2}^2=0$, so $\mu=0$; then $v=su_0$, and the second component
forces $s=0$. For surjectivity, given $(f,c)\in V^*\times\mathbb R$, set
$\mu:=-\langle f,u_0\rangle/\|u_0\|_{L^2}^2$, so that
$f+\mu u_0\in\operatorname{ran}\mathcal M_0$; choose $v_0$ with
$\mathcal M_0[v_0]=f+\mu u_0$ and put $v:=v_0+su_0$ with
$s:=\bigl(c/2-\int_\Omega u_0v_0\bigr)/\|u_0\|_{L^2}^2$. Then
$\mathcal L_0(v,\mu)=(f,c)$.

Thus $\mathcal L_0$ is an isomorphism, and since $\mathbb H$ is $C^1$ by
Lemma~\ref{lem:superposition}, Theorem~\ref{thm:implicit_function} produces a
$C^1$ map $h\mapsto(\widetilde u(h),\widetilde\lambda(h))$ with values in a
neighborhood $\mathcal W$ of $(u_0,\lambda_0)$, unique as a solution of
$\mathbb H=0$ in $\mathcal W$, and supplies $\omega_u$ through
\eqref{eq:IFT-modulus}. The pair $(u(h),\lambda_1(h))$ solves $\mathbb H=0$
and converges to $(u_0,\lambda_0)$ in $V\times\mathbb R$ by
Theorem~\ref{thm:compactness}, hence lies in $\mathcal W$ for $h$ near $h_0$,
and uniqueness there identifies the two branches.
\end{proof}

The following theorem identifies the linearized problem satisfied by the
derivative of the principal branch.

\begin{theorem}\label{thm:frechet}
Let $p\ge2$ and let $h_0\in\mathcal H_{\mathrm{ad}}$ satisfy
Assumption~\ref{ass:frechet-branch}. For $\vartheta\in C^1(\overline\gamma)$
put $u':=u'(h_0)[\vartheta]$ and $\lambda':=\lambda_1'(h_0)[\vartheta]$. Then
$u'\in V_\perp$ and, for every $\phi\in V$,
\begin{equation}\label{eq:linearized-weak}
\int_\Omega A_0\nabla u'\cdot\nabla\phi\,dx
-\lambda_0(p-1)\int_\Omega u_0^{\,p-2}u'\phi\,dx
+(p-1)\int_\gamma h_0u_0^{\,p-2}u'\phi\,d\sigma
=\lambda'\int_\Omega u_0^{\,p-1}\phi\,dx
-\int_\gamma\vartheta u_0^{\,p-1}\phi\,d\sigma .
\end{equation}
Formally, \eqref{eq:linearized-weak} together with
$\int_\Omega u_0^{\,p-1}u'\,dx=0$ is the linearized problem
\begin{equation}\label{eq:linearized_system}
\begin{cases}
-\operatorname{div}(A_0\nabla u')-\lambda_0(p-1)u_0^{\,p-2}u'
=\lambda'u_0^{\,p-1} & \text{in }\Omega,\\[1mm]
u'=0 & \text{on }\Gamma_D,\\[1mm]
A_0\nabla u'\cdot\nu+(p-1)h_0u_0^{\,p-2}u'=-\vartheta u_0^{\,p-1}
& \text{on }\gamma,\\[1mm]
\displaystyle\int_\Omega u_0^{\,p-1}u'\,dx=0 . &
\end{cases}
\end{equation}
\end{theorem}

\begin{proof}
By Lemma~\ref{lem:superposition} the map $\mathbb H$ is $C^1$, and by
Assumption~\ref{ass:frechet-branch} together with
Proposition~\ref{prop:lambda-remainder} the branch is Fr\'echet differentiable
at $h_0$. Differentiating $\mathbb H(h,u(h),\lambda_1(h))=(0,0)$ at $h_0$ in
the direction $\vartheta$ gives
$D_h\mathcal R[\vartheta]+\mathcal M_0[u']-\lambda'u_0^{\,p-1}=0$ in $V^*$
together with $p\int_\Omega u_0^{\,p-1}u'\,dx=0$. The second identity says
$u'\in V_\perp$, and expanding the first with \eqref{eq:DhR} and
\eqref{eq:DuR} gives \eqref{eq:linearized-weak}.
\end{proof}

As a consistency check, \eqref{eq:linearized-weak} reproduces the
unconditional formula \eqref{eq:lambda_derivative}. Taking $\phi=u_0$ and
using the symmetry of the form in \eqref{eq:DuR} with
$A_0\nabla u_0=(p-1)|\nabla u_0|^{p-2}\nabla u_0$, the three terms on the left
are $(p-1)$ times \eqref{eq:weak_formulation} for $u_0$ tested with $u'$, and
hence vanish; since $\int_\Omega u_0^{\,p}\,dx=1$, what remains is
$\lambda'=\int_\gamma\vartheta u_0^{\,p}\,d\sigma$, in agreement with
Proposition~\ref{prop:lambda-remainder}.

The system \eqref{eq:linearized_system} is read through
\eqref{eq:linearized-weak}; in particular $A_0\nabla u'\cdot\nu$ on $\gamma$
is the associated weak conormal trace, and no classical pointwise
interpretation is asserted for $u'$, whose equation has a coefficient
degenerating at the critical points of $u_0$ when $p>2$.
\subsection{Conditional stability}\label{sec:stability}
Let $p':=\frac{p}{p-1}$ and set
\begin{equation}\label{eq:Wdiv}
W_{\mathrm{div},p'}(\Omega)
:=
\bigl\{
g\in L^{p'}(\Omega;\mathbb R^n):
\operatorname{div}g\in L^{p'}(\Omega)
\bigr\},
\end{equation}
with graph norm
\[
\|g\|_{W_{\mathrm{div},p'}}
:=
\|g\|_{L^{p'}(\Omega)}
+
\|\operatorname{div}g\|_{L^{p'}(\Omega)}.
\]
For $g\in W_{\mathrm{div},p'}(\Omega)$, the normal trace is defined by
\begin{equation}\label{eq:normal-trace-pairing}
\langle g\cdot\nu,T\phi\rangle
:=
\int_\Omega g\cdot\nabla\phi\,dx
+
\int_\Omega\phi\,\operatorname{div}g\,dx,
\qquad
\phi\in W^{1,p}(\Omega).
\end{equation}
Since the trace operator maps $W^{1,p}(\Omega)$ onto
\[
W^{1-1/p,p}(\partial\Omega)
=
W^{1/p',p}(\partial\Omega),
\]
this defines a bounded normal-trace map
\[
W_{\mathrm{div},p'}(\Omega)
\longrightarrow
W^{-1/p',p'}(\partial\Omega).
\]
Notice that $1/p'=1-1/p$, which equals $1/p$ only when $p=2$.

For the measurement patch, define
\[
\widetilde W^{1/p',p}(\Gamma_{\mathrm m})
:=
\overline{C_c^\infty(\Gamma_{\mathrm m})}
^{\,W^{1/p',p}(\partial\Omega)}
\]
and
\[
W^{-1/p',p'}(\Gamma_{\mathrm m})
:=
\bigl(
\widetilde W^{1/p',p}(\Gamma_{\mathrm m})
\bigr)^*.
\]
Restriction to $\Gamma_{\mathrm m}$ is understood in this dual sense. Set
\[
Y
:=
\mathbb R\times W^{-1/p',p'}(\Gamma_{\mathrm m}),
\qquad
\|(a,q)\|_Y
:=
|a|+\|q\|_{W^{-1/p',p'}(\Gamma_{\mathrm m})}.
\]

\begin{definition}\label{def:measurement-map}
The \emph{measurement map} is
\begin{equation}\label{eq:measurement-map}
\mathcal F:\mathcal H_{\mathrm{ad}}\longrightarrow Y,
\qquad
\mathcal F(h)
:=
\bigl(
\lambda_1(h),q_h|_{\Gamma_{\mathrm m}}
\bigr),
\end{equation}
where $q_h$ is the weak normal trace of
$F(\nabla u(h))$.
\end{definition}

The next corollary transfers the differentiability of the principal branch to
the measurement map, which is the form used in the stability estimate.
\begin{corollary}\label{cor:measurement-differentiable}
Let $p\ge2$ and let $h_0\in\mathcal H_{\mathrm{ad}}$ satisfy
Assumption~\ref{ass:frechet-branch}. Set
\[
g(h):=F(\nabla u(h)).
\]
Then:
\begin{enumerate}[(i)]
\item
For every $h\in\mathcal H_{\mathrm{ad}}$,
\[
g(h)\in W_{\mathrm{div},p'}(\Omega),
\qquad
\operatorname{div}g(h)
=
-\lambda_1(h)u(h)^{p-1},
\]
so $\mathcal F(h)$ is well defined.

\item The map $h\mapsto g(h)$ is Fr\'echet differentiable at $h_0$ into
$W_{\mathrm{div},p'}(\Omega)$. For
\[
u'_\vartheta:=u'(h_0)[\vartheta],
\]
one has
\[
g'(h_0)[\vartheta]
=
A_0\nabla u'_\vartheta
\]
and
\[
\operatorname{div}g'(h_0)[\vartheta]
=
-\lambda_1'(h_0)[\vartheta]u_0^{p-1}
-\lambda_0(p-1)u_0^{p-2}u'_\vartheta
\quad\text{in }L^{p'}(\Omega).
\]
\item The measurement map is Fr\'echet differentiable at $h_0$, with
\begin{equation}\label{eq:DF-measurement}
D\mathcal F(h_0)[\vartheta]
=
\left(
\int_\gamma\vartheta u_0^p\,d\sigma,\,
\left.
(A_0\nabla u'_\vartheta)\cdot\nu
\right|_{\Gamma_{\mathrm m}}
\right).
\end{equation}
Moreover, there is a nondecreasing modulus $\widetilde\omega$ with
$\widetilde\omega(s)\to0$ as $s\downarrow0$ such that, for every
$h\in\mathcal U$,
\begin{equation}\label{eq:measurement-remainder}
\bigl\|
\mathcal F(h)-\mathcal F(h_0)
-D\mathcal F(h_0)[h-h_0]
\bigr\|_Y
\le
\widetilde\omega(t)t,
\qquad
t:=\|h-h_0\|_{C^1(\overline\gamma)}.
\end{equation}
\end{enumerate}
\end{corollary}

\begin{proof}
For~(i),
\[
|g(h)|^{p'}=|\nabla u(h)|^p\in L^1(\Omega),
\]
and the weak eigenvalue equation gives
\[
\operatorname{div}g(h)
=
-\lambda_1(h)u(h)^{p-1}.
\]
Since $(p-1)p'=p$, the right-hand side belongs to $L^{p'}(\Omega)$.

For~(ii), Assumption~\ref{ass:frechet-branch} gives differentiability of
\[
h\longmapsto\nabla u(h)
\quad\text{into }L^p(\Omega;\mathbb R^n).
\]
The vector-valued part of Lemma~\ref{lem:superposition} and the chain rule
therefore give
\[
g'(h_0)[\vartheta]
=
DF(\nabla u_0)\nabla u'_\vartheta
=
A_0\nabla u'_\vartheta
\quad\text{in }L^{p'}(\Omega;\mathbb R^n).
\]
Likewise, Proposition~\ref{prop:lambda-remainder}, the scalar part of
Lemma~\ref{lem:superposition}, and the product rule show that
\[
h\longmapsto
-\lambda_1(h)u(h)^{p-1}
\]
is Fr\'echet differentiable into $L^{p'}(\Omega)$ with the stated
derivative. Hence $h\mapsto g(h)$ is differentiable in the graph norm of
$W_{\mathrm{div},p'}(\Omega)$.

 For (iii), the two components of $D\mathcal F(h_0)$ are obtained by composing
(ii) with the bounded normal-trace and restriction maps, and by (59) for the
eigenvalue component. It remains to produce the modulus in (90). Write
$\vartheta:=h-h_0$, $z_h:=\nabla u(h)$, $z_0:=\nabla u_0$. By (82),
\[
\|u(h)-u_0\|_{W^{1,p}(\Omega)}\le\bigl(\|u'(h_0)\|+\omega_u(t)\bigr)t=:C_1t,
\]
with $C_1$ independent of $h\in\mathcal U$ after shrinking $\mathcal U$ to a
bounded neighborhood. For the flux component,
\[
g(h)-g(h_0)-A_0\nabla u'_\vartheta
=\bigl[\mathcal N(z_h)-\mathcal N(z_0)-DF(z_0)(z_h-z_0)\bigr]
+DF(z_0)\bigl(z_h-z_0-\nabla u'_\vartheta\bigr).
\]
By Lemma 4.13 the first bracket is bounded in $L^{p'}(\Omega)$ by
$\omega_{\mathcal N}(\|z_h-z_0\|_{L^p})\|z_h-z_0\|_{L^p}
\le\omega_{\mathcal N}(C_1t)C_1t$ for a nondecreasing modulus
$\omega_{\mathcal N}$ vanishing at the origin, and the second term is bounded
by $(p-1)\|z_0\|_{L^p}^{p-2}\omega_u(t)t$ by (77) and (82). For the divergence
component,
\[
\begin{aligned}
&\operatorname{div}\bigl(g(h)-g(h_0)-A_0\nabla u'_\vartheta\bigr)\\
&\quad=-\bigl(\lambda_1(h)-\lambda_0-\lambda_1'(h_0)[\vartheta]\bigr)u_0^{p-1}
-\bigl(\lambda_1(h)-\lambda_0\bigr)(p-1)u_0^{p-2}u'_\vartheta\\
&\qquad-\lambda_1(h)\bigl[u(h)^{p-1}-u_0^{p-1}
-(p-1)u_0^{p-2}u'_\vartheta\bigr].
\end{aligned}
\]
The first term is at most $\omega_\lambda(t)t\,\|u_0^{p-1}\|_{L^{p'}}$ by
Proposition 4.6. For the second, H\"older with the exponents $p/(p-2)$ and $p$
gives $\|u_0^{p-2}u'_\vartheta\|_{L^{p'}}\le
\|u_0\|_{L^p}^{p-2}\|u'_\vartheta\|_{L^p}\le Ct$, while
$|\lambda_1(h)-\lambda_0|\le Ct$ by Proposition 4.2, so this term is
$O(t^2)$. The bracket in the third term splits as the flux component did, into
the superposition remainder of $s\mapsto s^{p-1}$ at $u_0$ and
$(p-1)u_0^{p-2}(u(h)-u_0-u'_\vartheta)$, bounded respectively by
$\omega_{\mathcal N}(C_1t)C_1t$ and $C\omega_u(t)t$; and $\lambda_1(h)$ is
bounded on $\mathcal U$. Adding these bounds, dividing by $t$ and taking the
supremum over $0<t\le s$ defines a nondecreasing $\widetilde\omega$ with
$\widetilde\omega(s)\to0$, for which (90) holds.
\end{proof}
Before stating the stability estimate we record what can be proved
unconditionally on a finite-dimensional set of directions, where injectivity
of the linearized measurement map already yields an estimate of the required
interpolation form.

\begin{lemma}\label{lem:finite-dimensional-stability}
Suppose that $\mathcal F:\mathcal H_{\mathrm{ad}}\to Y$ is Fr\'echet
differentiable at $h_0$, let $\{0\}\neq X_m\subset C^1(\overline\gamma)$ be
finite dimensional, and assume that
$D\mathcal F(h_0)|_{X_m}:X_m\to Y$ is injective. Then there exists $C_m>0$
such that
\begin{equation}\label{eq:finite-dimensional-stability}
\|\vartheta\|_{L^2(\gamma)}
\le
C_m\|D\mathcal F(h_0)[\vartheta]\|_Y,
\qquad\vartheta\in X_m .
\end{equation}
Moreover, for every $\alpha\in(0,1)$ there exists $C_{m,\alpha}>0$ such that
\begin{equation}\label{eq:finite-dimensional-interpolation}
\|\vartheta\|_{L^2(\gamma)}
\le
C_{m,\alpha}
\|D\mathcal F(h_0)[\vartheta]\|_Y^\alpha
\|\vartheta\|_{C^1(\overline\gamma)}^{1-\alpha},
\qquad\vartheta\in X_m .
\end{equation}
\end{lemma}

\begin{proof}
Elements of $X_m$ are continuous on $\overline\gamma$, so
$\|\vartheta\|_{L^2(\gamma)}=0$ forces $\vartheta=0$; thus
$\|\cdot\|_{L^2(\gamma)}$ is a norm on $X_m$ and, all norms on $X_m$ being
equivalent, the sphere
\[
S_m:=\{\vartheta\in X_m:\|\vartheta\|_{L^2(\gamma)}=1\}
\]
is compact. The map $\vartheta\mapsto\|D\mathcal F(h_0)[\vartheta]\|_Y$ is
continuous on $X_m$ and, by injectivity, positive on $S_m$, so
\[
\sigma_m:=\min_{\vartheta\in S_m}\|D\mathcal F(h_0)[\vartheta]\|_Y>0,
\]
and homogeneity gives \eqref{eq:finite-dimensional-stability} with
$C_m=\sigma_m^{-1}$.

Let $C_\gamma$ be the norm of the embedding
$C^1(\overline\gamma)\hookrightarrow L^2(\gamma)$. For $\alpha\in(0,1)$,
\[
\|\vartheta\|_{L^2(\gamma)}
=\|\vartheta\|_{L^2(\gamma)}^{\alpha}\|\vartheta\|_{L^2(\gamma)}^{1-\alpha}
\le
C_m^\alpha C_\gamma^{1-\alpha}
\|D\mathcal F(h_0)[\vartheta]\|_Y^\alpha
\|\vartheta\|_{C^1(\overline\gamma)}^{1-\alpha},
\]
which is \eqref{eq:finite-dimensional-interpolation}.
\end{proof}

Lemma~\ref{lem:finite-dimensional-stability} does not verify the hypothesis
below: \eqref{eq:finite-dimensional-interpolation} holds only on $X_m$, and
its constant is not controlled as $\dim X_m$ grows. The infinite-dimensional
statement is a genuine additional input, which we record as a hypothesis.

\begin{assumption}\label{ass:linearized-stability}
Let $h_0\in\mathcal H_{\mathrm{ad}}$ satisfy
Assumption~\ref{ass:frechet-branch}. There exist $C_0>0$ and
$\alpha\in(0,1)$ such that
\begin{equation}\label{eq:linearized-stability}
\|\vartheta\|_{L^2(\gamma)}
\le
C_0
\|D\mathcal F(h_0)[\vartheta]\|_Y^\alpha
\|\vartheta\|_{C^1(\overline\gamma)}^{1-\alpha}
\end{equation}
for every $\vartheta\in C^1(\overline\gamma)$.
\end{assumption}

\begin{remark}
Assumption~\ref{ass:linearized-stability} is weaker than the existence of a
bounded left inverse for $D\mathcal F(h_0)$: the interpolation structure of
\eqref{eq:linearized-stability} reflects the severe ill-posedness, in that a
small measurement residual controls only a weak norm of $\vartheta$ once a
stronger norm is bounded a priori. For $p=2$ the linearized operator is the
Laplacian with a Robin condition, and quantitative unique continuation
results are available for related inverse Robin problems; the known estimates
are of logarithmic type \cite{Chaabane04} and are not stated for the spectral
measurement map, so they do not yield \eqref{eq:linearized-stability}. For
$p\neq2$ the corresponding estimate would have to be established for
$\operatorname{div}(A_0\nabla\,\cdot\,)$, whose coefficient degenerates at the
critical points of $u_0$ by \eqref{eq:DF-two-sided}. In both cases
\eqref{eq:linearized-stability} is therefore kept as an explicit hypothesis.
\end{remark}

The conditional stability estimate now follows from the remainder bound and
the assumed linearized estimate.

\begin{theorem}\label{thm:stability}
Let $p\ge2$ and let $h_0\in\mathcal H_{\mathrm{ad}}$ satisfy
Assumptions~\ref{ass:frechet-branch} and \ref{ass:linearized-stability}. Fix
$M>0$. Then there exists $\rho_0>0$ such that, for every
$h\in\mathcal H_{\mathrm{ad}}$ with
\[
t:=\|h-h_0\|_{C^1(\overline\gamma)}\le\min\{M,\rho_0\},
\]
one has
\begin{equation}\label{eq:stability-estimate}
\|h-h_0\|_{L^2(\gamma)}
\le
C_0M^{1-\alpha}\bigl(\delta+\widetilde\omega(t)t\bigr)^\alpha,
\qquad
\delta:=\|\mathcal F(h)-\mathcal F(h_0)\|_Y .
\end{equation}
\end{theorem}

\begin{proof}
Choose $\rho_0>0$ with $B_{C^1}(h_0,\rho_0)\subset\mathcal U$ and set
$\vartheta:=h-h_0$. By \eqref{eq:measurement-remainder} and the triangle
inequality,
\[
\|D\mathcal F(h_0)[\vartheta]\|_Y\le\delta+\widetilde\omega(t)t .
\]
Substituting into \eqref{eq:linearized-stability} and using $t\le M$ gives
\eqref{eq:stability-estimate}.
\end{proof}

The right-hand side of \eqref{eq:stability-estimate} involves the unknown $t$
through the remainder $\widetilde\omega(t)t$, so the estimate is conditional
rather than a data-only H\"older bound. On a finite-dimensional set of
directions the remainder can be absorbed, and the estimate becomes a genuine
data-only Lipschitz bound.

\begin{corollary}\label{cor:finite-dimensional-nonlinear-stability}
Let $p\ge2$, let $h_0\in\mathcal H_{\mathrm{ad}}$ satisfy
Assumption~\ref{ass:frechet-branch}, and let
$\{0\}\neq X_m\subset C^1(\overline\gamma)$ be finite dimensional with
$D\mathcal F(h_0)|_{X_m}$ injective. Then there exist $r_m>0$ and $C_m>0$,
with $C_m$ as in \eqref{eq:finite-dimensional-stability}, such that
\begin{equation}\label{eq:finite-dimensional-nonlinear-stability}
\|h-h_0\|_{L^2(\gamma)}
\le
2C_m\|\mathcal F(h)-\mathcal F(h_0)\|_Y
\end{equation}
whenever
\[
h\in(h_0+X_m)\cap\mathcal H_{\mathrm{ad}},
\qquad
\|h-h_0\|_{C^1(\overline\gamma)}\le r_m .
\]
\end{corollary}

\begin{proof}
Set $\vartheta:=h-h_0\in X_m$ and $t:=\|\vartheta\|_{C^1(\overline\gamma)}$.
By \eqref{eq:finite-dimensional-stability},
$\|D\mathcal F(h_0)[\vartheta]\|_Y\ge C_m^{-1}\|\vartheta\|_{L^2(\gamma)}$,
and since all norms on $X_m$ are equivalent there is $K_m>0$ with
$t\le K_m\|\vartheta\|_{L^2(\gamma)}$. Hence, by
\eqref{eq:measurement-remainder},
\[
\|\mathcal F(h)-\mathcal F(h_0)\|_Y
\ge
\|D\mathcal F(h_0)[\vartheta]\|_Y-\widetilde\omega(t)t
\ge
\bigl(C_m^{-1}-K_m\widetilde\omega(t)\bigr)\|\vartheta\|_{L^2(\gamma)} .
\]
Choose $r_m>0$ so that $B_{C^1}(h_0,r_m)\subset\mathcal U$ and
$K_m\widetilde\omega(r_m)\le\tfrac1{2C_m}$. For $t\le r_m$ the monotonicity of
$\widetilde\omega$ then gives
$\|\mathcal F(h)-\mathcal F(h_0)\|_Y\ge\tfrac1{2C_m}\|\vartheta\|_{L^2(\gamma)}$,
which is \eqref{eq:finite-dimensional-nonlinear-stability}.
\end{proof}

Thus, under Assumption~\ref{ass:frechet-branch}, finite-dimensional
injectivity yields local data-only Lipschitz stability, whereas
Theorem~\ref{thm:stability} remains conditional in the full coefficient class.

\subsection{Status and verification of the structural hypotheses}
\label{sec:verification}

The conditional results of Sections~\ref{sec:uniqueness}--\ref{sec:stability}
use three structural assumptions: Assumption~\ref{ass:connectivity} on the
propagation set, Assumption~\ref{ass:frechet-branch} on differentiability of
the principal eigenfunction branch, and
Assumption~\ref{ass:linearized-stability} on the linearized inverse map. In
addition, when $p>2$ and $n\ge3$, Theorem~\ref{thm:uniqueness} uses the local
$C^{2,\alpha_0}$ boundary regularity stated there. We summarize below what is
automatic, what follows from explicit sufficient conditions, and what remains
conditional.

\paragraph{Boundary geometry.}
The standing hypothesis \eqref{eq:geometric-hypothesis} is automatic when
$\Gamma_D$ and $\gamma$ are unions of connected components of
$\partial\Omega$ whose union is $\partial\Omega$, because then
$\zeta=\emptyset$. This includes the spherical shell considered below, for
which Lemma~\ref{lem:regularity}(iv) gives global $C^{1,\beta}$ regularity up
to the boundary. The additional $C^{2,\alpha_0}$ regularity near
$\overline{\Gamma_{\mathrm m}}$ assumed in Theorem~\ref{thm:uniqueness} when
$p>2$ and $n\ge3$ is used only to obtain Lipschitz regularity of
$\overline A$ on the boundary collar in the proof of that theorem, and it is
automatic for the smooth shell and ellipsoid geometries considered below.

\subsubsection*{Assumption~\ref{ass:connectivity}}

For $p=2$ the assumption is vacuous because $D=\Omega$. Let $p>2$ and set
\[
\Sigma_*:=
\begin{cases}
\Sigma_\cap, & n=2,\\
\Sigma_\cup, & n\ge3.
\end{cases}
\]
By Lemma~\ref{lem:critical-sets} the set $\Sigma_*$ is relatively closed in
$\Omega$, so Lemma~\ref{lem:separator} shows that the sufficient condition
\[
\mathcal H^{n-1}(\Sigma_*)=0
\]
implies that $D=\Omega\setminus\Sigma_*$ is connected. Moreover, by
Lemma~\ref{lem:path-matrix} the ellipticity of $\overline A$ on a compact
$K\Subset D$ is controlled from below by $\min_KM$ and from above by
$\max_KM$, both positive and finite. Thus the ellipticity required for
propagation is local and quantitative, and no global positive lower bound for
the gradients is imposed.

The next remark isolates the case $h_1=h_2$. Uniqueness is vacuous there, but
it is the sharpest test of how restrictive the propagation hypothesis is, and
the shell model below shows that it already fails on the diagonal when $p>2$.

\begin{remark}\label{rem:diagonal}
For a function $u$, write
$\Sigma(u):=\{x\in\Omega:\nabla u(x)=0\}$. For $p>2$, on the diagonal
$h_1=h_2=h$ one has $\Sigma_\cap=\Sigma_\cup=\Sigma(u(h))$, so
Assumption~\ref{ass:connectivity} reduces to
\begin{equation}\label{eq:diagonal-condition}
\Omega\setminus\Sigma(u(h))
\quad\text{is connected}.
\end{equation}
This permits critical points, provided their set does not separate $\Omega$;
by Lemma~\ref{lem:separator} a sufficient condition is
$\mathcal H^{n-1}(\Sigma(u(h)))=0$. No diagonal critical-set condition is
needed when $p=2$.
\end{remark}

The following shell model shows that \eqref{eq:diagonal-condition} can fail
for $p>2$ even in a smooth geometry.

\begin{proposition}\label{prop:annulus}
Let $n\ge2$, let $0<r_{\mathrm{in}}<r_{\mathrm{out}}$, and set
\[
\Omega:=\{x\in\mathbb R^n:r_{\mathrm{in}}<|x|<r_{\mathrm{out}}\},
\qquad
\Gamma_D:=\{|x|=r_{\mathrm{in}}\},
\qquad
\gamma:=\{|x|=r_{\mathrm{out}}\} .
\]
Let $p\ge2$ and let $h>0$ be constant. Then the normalized positive principal
eigenfunction is radial, $u(x)=U(|x|)$, and there is a unique
$r^*\in(r_{\mathrm{in}},r_{\mathrm{out}})$ such that
$\Sigma(u)=\{x\in\Omega:|x|=r^*\}$.
\end{proposition}

\begin{proof}
The data are invariant under rotations, so for every rotation $R$ the
function $u\circ R$ is again a normalized positive principal eigenfunction;
by \eqref{eq:full-simplicity} it equals $u$, and $u$ is radial. Set
\[
Q(r):=r^{n-1}|U'(r)|^{p-2}U'(r) .
\]
The radial form of the equation shows that $Q$ is absolutely continuous with
\[
Q'(r)=-\lambda_1(h)r^{n-1}U(r)^{p-1}<0,
\]
so $Q$ is continuous and strictly decreasing. At the inner boundary the
outward normal is $-x/|x|$, so Lemma~\ref{lem:hopf} gives
$\partial_\nu u=-U'(r_{\mathrm{in}})<0$ and hence
$Q(r_{\mathrm{in}})>0$. At the outer boundary the Robin condition gives
\[
|U'(r_{\mathrm{out}})|^{p-2}U'(r_{\mathrm{out}})
=-h\,U(r_{\mathrm{out}})^{p-1}<0,
\]
so $Q(r_{\mathrm{out}})<0$. Thus $Q$ has exactly one zero $r^*$, and since
$r^{n-1}>0$ one has $Q(r)=0$ if and only if $U'(r)=0$.
\end{proof}

\begin{remark}\label{rem:annulus-scope}
For $p>2$ the critical sphere of Proposition~\ref{prop:annulus} separates
$\Omega$, so \eqref{eq:diagonal-condition} fails. For $p=2$
Assumption~\ref{ass:connectivity} is vacuous and the critical sphere creates
no obstruction to Theorem~\ref{thm:uniqueness}. In either case constant Robin
coefficients are ordered, so Theorem~\ref{thm:ordered-uniqueness} already
gives uniqueness from the principal eigenvalue alone. The proposition
therefore delimits the scope of the propagation hypothesis rather than
exhibiting a failure of uniqueness.
\end{remark}

\begin{remark}
\label{rem:order-of-vanishing}
\label{rem:two-difficulties}
The shell also explains the regularity distinction between $\Sigma_\cap$ and
$\Sigma_\cup$. Since $Q(r^*)=0$ and
$Q'(r^*)=-\lambda_1(h)(r^*)^{n-1}U(r^*)^{p-1}<0$, one has
$Q(r)=-c_0(r-r^*)+o(|r-r^*|)$ with $c_0>0$, and therefore
\[
|U'(r)|=\left(\frac{|Q(r)|}{r^{n-1}}\right)^{1/(p-1)}
\asymp|r-r^*|^{1/(p-1)} .
\]
For $p>2$ the gradient is not Lipschitz across the critical sphere, and for
$r\neq r^*$,
\[
|U''(r)|\asymp|r-r^*|^{(2-p)/(p-1)} .
\]

A path average may consequently remain uniformly elliptic while inheriting
sub-Lipschitz endpoint regularity. This is the configuration in which the two
eigenfunctions have critical spheres at different radii, so that along the
segment one endpoint stays bounded away from zero while the other vanishes.
Already in the scalar model $p=3$, where $F(s)=|s|s$ and $DF(s)=2|s|$, one has
for $b=1$ and $0\le a\le1$
\[
\int_0^1DF\bigl((1-t)b+ta\bigr)\,dt=1+a ,
\]
so that the choice $a(r)=|r-r^*|^{1/2}$, which is the order of vanishing
computed above for $p=3$, produces a coefficient that is uniformly positive
but not Lipschitz.

Accordingly, on $\Sigma_\cup\setminus\Sigma_\cap$ the path matrix remains
locally uniformly elliptic but may lack Lipschitz regularity, whereas on
$\Sigma_\cap$ ellipticity itself degenerates. In dimension two the first
obstruction is removed by Remark~\ref{rem:ucp-2d}; in dimensions $n\ge3$ the
propagation argument therefore avoids $\Sigma_\cup$.
\end{remark}

The measure condition above is an explicit sufficient criterion whenever the
critical set is known. This paper does not establish such a critical-set
estimate for general nonsymmetric mixed-boundary configurations.

\subsubsection*{Assumption~\ref{ass:frechet-branch}}

For $p=2$, Proposition~\ref{prop:frechet-p2} proves the assumption at every
$h_0\in\mathcal H_{\mathrm{ad}}$, by the implicit function theorem and
simplicity of the principal eigenvalue.

For $p>2$ differentiability of the principal branch is not proved here, and
Remark~\ref{rem:no-isomorphism} shows that the standard implicit-function
argument cannot be carried out in the natural pair $V\times\mathbb R$ and
$V^*\times\mathbb R$. All results using
Assumption~\ref{ass:frechet-branch} therefore remain conditional for $p>2$.

\subsubsection*{Assumption~\ref{ass:linearized-stability}}

In the full infinite-dimensional coefficient class,
\eqref{eq:linearized-stability} is retained as an explicit hypothesis. For
$p=2$, logarithmic stability results for related inverse Robin problems, such
as \cite{Chaabane04}, motivate conditional estimates of this kind but do not
imply the power-type inequality for the spectral measurement map used here.
For $p>2$ the principal coefficient of the linearized operator degenerates on
the critical set of $u_0$, and no estimate of the form
\eqref{eq:linearized-stability} is derived in this paper; the cutoff
regularization of Remark~\ref{rem:localized-regularization} does not help,
because the regularized operator is not the derivative of the forward map.

For $p=2$, however, one ingredient of the stability analysis is available
unconditionally: the linearized measurement map is injective, which is
precisely the hypothesis that the finite-dimensional stability mechanism
requires.
 
\begin{proposition}
Let $p=2$ and let $h_0\in\mathcal H_{\mathrm{ad}}$. Then
$D\mathcal F(h_0):C^1(\overline\gamma)\to Y$ is injective.
\end{proposition}
 
\begin{proof}
By Proposition~\ref{prop:frechet-p2}, Assumption~\ref{ass:frechet-branch}
holds, so $D\mathcal F(h_0)$ is defined by
Corollary~\ref{cor:measurement-differentiable}. Let
$D\mathcal F(h_0)[\vartheta]=0$ and put $u':=u'(h_0)[\vartheta]$.
 
The first component of \eqref{eq:DF-measurement} gives
$\lambda':=\lambda_1'(h_0)[\vartheta]=0$. Since $A_0\equiv I$ for $p=2$,
Corollary~\ref{cor:measurement-differentiable}(ii) shows that
\[
g':=g'(h_0)[\vartheta]=\nabla u'\in W^{\operatorname{div},2}(\Omega),
\qquad
\operatorname{div} g'=-\lambda' u_0-\lambda_0 u'=-\lambda_0 u'
\quad\text{in }L^2(\Omega),
\]
that is, $u'\in H^1(\Omega)$ solves the homogeneous equation
$-\Delta u'=\lambda_0 u'$ in $\Omega$.
 
We next convert the flux datum into the local weak formulation required by
Theorem~\ref{thm:boundary_ucp_local}. 
Note that \eqref{eq:linearized-weak} admits only test functions vanishing on
$\Gamma_D$ and therefore carries no information about the conormal trace on
$\Gamma_{\mathrm m}$; that information enters through the second component of
\eqref{eq:DF-measurement}, which states that the weak normal trace
$g'\cdot\nu$ vanishes in $W^{-1/2,2}(\Gamma_{\mathrm m})$, that is,
\[
\langle g'\cdot\nu,\psi\rangle=0
\qquad\text{for every }\psi\in\widetilde W^{1/2,2}(\Gamma_{\mathrm m}).
\]
Choose a connected open set $U\subset\mathbb R^n$ with
$\emptyset\neq U\cap\partial\Omega\subset\Gamma_{\mathrm m}$. For
$\phi\in C_c^\infty(U)$ the trace $T\phi$ is smooth and compactly supported in
$\Gamma_{\mathrm m}$, hence belongs to
$\widetilde W^{1/2,2}(\Gamma_{\mathrm m})$, and the defining identity of the
weak normal trace gives 
\[
0=\langle g'\cdot\nu, T\phi\rangle
=\int_{U\cap\Omega}\nabla u'\cdot\nabla\phi\,dx
 +\int_{U\cap\Omega}\phi\,\operatorname{div} g'\,dx
=\int_{U\cap\Omega}\nabla u'\cdot\nabla\phi\,dx
 -\lambda_0\int_{U\cap\Omega}u'\phi\,dx .
\]
Since moreover $u'\in V$, so that $u'=0$ on
$U\cap\partial\Omega\subset\Gamma_D$ in the sense of traces, the pair of
conditions above is exactly the vanishing-Cauchy-data hypothesis of
Theorem~\ref{thm:boundary_ucp_local}, with $A=I$ and $B=1$. Hence
$u'\equiv0$ in $U\cap\Omega$, and Theorem~\ref{thm:interior_propagation},
applied on the connected set $\Omega$, gives $u'\equiv0$ in $\Omega$.
 
With $u'=0$ and $\lambda'=0$, \eqref{eq:linearized-weak} reduces to
\[
\int_\gamma\vartheta\,u_0\,\phi\,d\sigma=0
\qquad\text{for all }\phi\in V .
\]
Every $\phi\in C_c^\infty(\gamma)$ extends to an element of $V$, so the traces
on $\gamma$ of functions in $V$ are dense in $L^2(\gamma)$, whence
$\vartheta u_0=0$ on $\gamma$. Since $u_0>0$ on $\gamma$ by
Proposition~\ref{prop:boundary_nondegen}, it follows that $\vartheta=0$.
\end{proof}
 
Consequently, for $p=2$ and every finite-dimensional
$X_m\subset C^1(\overline\gamma)$, the restriction
$D\mathcal F(h_0)|_{X_m}$ is injective, and
Corollary~\ref{cor:finite-dimensional-nonlinear-stability} gives local
data-only Lipschitz stability on $(h_0+X_m)\cap\mathcal H_{\mathrm{ad}}$. This
is unconditional. Injectivity alone does not yield
\eqref{eq:linearized-stability}, so the infinite-dimensional estimate remains
conditional even for $p=2$. For $p>2$ the branch hypothesis is not available
and neither statement applies.

\paragraph{A non-radial domain.}
Let
\[
\Omega=\Bigl\{x\in\mathbb R^n:\sum_{j=1}^n\frac{x_j^2}{a_j^2}<1\Bigr\},
\qquad a_1\neq a_2 ,
\]
and choose nonempty relatively open parts $\Gamma_D$ and $\gamma$ satisfying
the standing decomposition \eqref{eq:geometric-hypothesis}. This is a smooth,
genuinely non-radial multidimensional geometry, and every hypothesis used in
the preceding two subsubsections at $p=2$ holds on it.

Let $h_*>0$ be constant, let $0\not\equiv\psi\in C_c^1(\gamma)$ satisfy
$\psi\ge0$, and let $I\subset\mathbb R$ be an open interval such that
$h_\theta:=h_*+\theta\psi\in\mathcal H_{\mathrm{ad}}$ for $\theta\in I$. By
Lemma~\ref{lem:one-parameter-family},
\[
\lambda_1'(h_\theta)[\psi]=\int_\gamma\psi\,u_\theta^{\,p}\,d\sigma>0 ,
\]
so the principal eigenvalue alone determines $\theta$, with inverse Lipschitz
stability on every compact interval $J\Subset I$. This holds for every
$1<p<\infty$ and uses none of the structural assumptions.

For $p=2$ one has $\overline A\equiv I$, so
Assumption~\ref{ass:connectivity} is automatic, and
Proposition~\ref{prop:frechet-p2} gives differentiability of the principal
branch; by the proposition above, $D\mathcal F(h_0)$ is injective at every
$h_0\in\mathcal H_{\mathrm{ad}}$, and
Corollary~\ref{cor:finite-dimensional-nonlinear-stability} applies on every
finite-dimensional affine class. For $p>2$ the example verifies none of
critical-set connectivity, principal-branch differentiability, or
quantitative linearized stability, which retain the conditional status
described above.
\section*{Acknowledgments}
The authors are grateful to Vladimir Bobkov for carefully reading the
manuscript and for his insightful suggestions and helpful discussions.


\end{document}